\RequirePackage{fix-cm}
\documentclass{amsart} 
\usepackage[utf8]{inputenc}
\usepackage{mathptmx}
\usepackage{latexsym}
\usepackage{amsmath}
\usepackage{amssymb}
\usepackage{amsfonts} 
\usepackage{eucal} 
\usepackage{amsbsy}
\usepackage{amsthm}
\usepackage[all]{xy}
\usepackage{hyperref}
\usepackage{graphicx}
\usepackage[section]{algorithm}
\usepackage{algpseudocode}

\usepackage{calligra,mathrsfs}

\newtheorem{thm}{Theorem}[section]
\newtheorem*{thm*}{Theorem}
\newtheorem{lemma}[thm]{Lemma}
\newtheorem{prop}[thm]{Proposition}
\newtheorem{cor}[thm]{Corollary}
\newtheorem*{cor*}{Corollary}

\theoremstyle{definition}
\newtheorem{defn}[thm]{Definition}
\newtheorem{example}[thm]{Example}

\theoremstyle{remark}
\newtheorem{remark}[thm]{Remark}

\newcommand {\Ta}    {\ensuremath{\mathcal{T}}}
\newcommand {\Xa}    {\ensuremath{\mathcal{X}}}

\newcommand {\real}  {\ensuremath{\mathbb{R}}}
\newcommand {\intg}  {\ensuremath{\mathbb{Z}}}

\newcommand {\codim} {\ensuremath{\operatorname{codim}}}

\newcommand {\smlhf} {\ensuremath{\mbox{$\frac{1}{2}$}}}

\newcommand {\cone}  {\ensuremath{\operatorname{cone}}}

\newcommand {\pt}    {\ensuremath{\operatorname{pt}}}
\newcommand {\id}    {\ensuremath{\operatorname{id}}}

\newcommand {\Princ}  {\ensuremath{\operatorname{Princ}}}

\newcommand {\PL}   {{\ensuremath{\operatorname{PL}}}}

\newcommand {\IM}   {{\ensuremath{\operatorname{IM}}}}

\begin{document}

\title{Stratified Formal Deformations and 
  Intersection Homology of Data Point Clouds}

\author{Markus Banagl}

\address{Mathematisches Institut, Universit\"at Heidelberg,
  Im Neuenheimer Feld 205, 69120 Heidelberg, Germany}

\email{banagl@mathi.uni-heidelberg.de}

\thanks{This work is supported by Deutsche Forschungsgemeinschaft 
  (DFG, German Research Foundation) under 
   Germany’s Excellence Strategy EXC-2181/1 - 390900948 
   (the Heidelberg STRUCTURES Excellence Cluster).}

\date{May, 2020}

\author{Tim M\"ader}

\address{Mathematisches Institut, Universit\"at Heidelberg,
  Im Neuenheimer Feld 205, 69120 Heidelberg, Germany}

\email{tmaeder@mathi.uni-heidelberg.de}

\author{Filip Sadlo}

\address{Institut f\"ur wissenschaftliches Rechnen, Universit\"at Heidelberg,
  Im Neuenheimer Feld 205, 69120 Heidelberg, Germany}

\email{sadlo@uni-heidelberg.de}



\begin{abstract}
Intersection homology is a topological invariant which
detects finer information in a space than ordinary homology.
Using ideas from classical simple homotopy theory,
we construct local combinatorial transformations
on simplicial complexes under which intersection homology
remains invariant. In particular, we obtain the notions of
stratified formal deformations and 
stratified spines of a complex, leading to reductions of
complexes prior to computation of intersection homology.  
We implemented the algorithmic execution
of such transformations, as well as the calculation of
intersection homology, and apply these algorithms to
investigate the intersection homology of stratified spines
in Vietoris-Rips type complexes associated to point sets
sampled near given, possibly singular, spaces.
\end{abstract}

\maketitle

\tableofcontents

\section{Introduction}

In \cite{bendichharer}, Bendich and Harer proposed the use of intersection
homology, a finer invariant than ordinary homology, in the
topological analysis of large data sets. In particular, they
defined persistent intersection homology and described its
algorithmic computation. Bendich and Harer conclude by asking whether
in addition to persistence considerations and the selection criteria 
imposed by intersection
homology, other simplex removal processes might be conducive
to a finer understanding of large data sets.
In the present paper, we adapt ideas from classical simple homotopy theory
(\cite{whitehead}, \cite{cohen}), due to J. H. C. Whitehead,
to the stratified setting in order to construct 
local combinatorial operations on a simplicial complex
under which
intersection homology provably remains invariant.
These operations lead in particular to a class of stratified spines 
of a complex that
are algorithmically computable, generally not homeomorphic to each other, yet
possess isomorphic intersection homology groups. 
Spines are well-known to carry deep information on the topological
complexity of spaces, sometimes even allowing for classification
results, see e.g. \cite{matveev}. 
Carrying over algorithmic simplicial reduction processes to filtered
and stratified settings is in line with the general objective 
in topological data science of
finding ways to reduce the size of complexes (while preserving
topological invariants), before computing homology.
Zomorodian \cite{zomorodian}, for example, 
underscores the advantage of prior reduction
by citing the 
supercubical complexity in the size of the complex of
classical matrix reduction algorithms for computing homology
over the integers, and the quadratic space and cubic time complexity
of Gaussian elimination over field coefficients.

The ordinary homology of a topological space is invariant
under homotopy equivalences. If the space is a closed oriented
manifold, then the ordinary homology satisfies in addition
Poincar\'e duality, relating the groups in complementary dimensions.
The homotopy invariance is of course a fundamental general pillar of
the theory, but allows in particular for applications of homology
in data science, since the simplicial complexes generated from 
data point clouds are usually only homotopy equivalent to an
underlying space whose structure one hopes to detect.
General homotopy equivalences in high dimensions are
very complicated, both conceptually and algorithmically.
This was realized early on in the development of topology
and prompted research on more combinatorial models of homotopy
equivalences. 
In his landmark paper \cite{whitehead}, J. H. C. Whitehead introduced
the notion of simple homotopy equivalence, which for simplicial
complexes can be described entirely combinatorially in terms
of local operations: elementary simplicial collapses and expansions.
This notion together with Whitehead's torsion became very influential 
in high dimensional manifold topology, yielding such results as the
$s$-cobordism theorem. 

Poincar\'e duality on the other hand is the source of numerous invariants 
that play important roles in understanding and classifying manifolds even 
within their simple homotopy type.
However, in real world applications, spaces often possess strata of
non-manifold, i.e. singular, points. Simple examples show that the
presence of such singularities generally invalidates Poincar\'e duality.
Several solutions to this problem exist:
Goresky-MacPherson's intersection homology \cite{gm1}, \cite{gm2},
Cheeger's $L^2$-cohomology \cite{cheeger1}, \cite{cheeger2},
\cite{cheeger3} and the theory of
intersection spaces, \cite{banaglintsp}, due to the first
named author.
The present paper is concerned only with intersection homology.

Now, contrary to ordinary homology, intersection homology
is not invariant under
arbitrary homotopy equivalences, even when these are simple in 
Whitehead's sense. On the other hand, stratum preserving
homotopy equivalences will not alter intersection homology, even
when they are not simple.
A natural question is thus:
Do there exist local simplicial moves that are akin to Whitehead's, but do
preserve given strata and intersection homology?
We report here on the discovery of such moves;
we call them \emph{stratified collapses and expansions}.
In the manifold situation with only one stratum, these coincide
with Whitehead's operations.
The stratified collapses and expansions defined here 
may ultimately lead to a concept of stratified simple homotopy type,
which we shall not fully develop in the present paper.
Classically, finite sequences of elementary collapses and expansions
are called formal deformations. We thus define stratified formal
deformations as finite sequences of stratified collapses and expansions.
Our central theoretical result, Theorem \ref{thm.main}, states that
stratified formal deformations (assumed to preserve formal codimensions 
of strata) are stratified homotopy equivalences in the classical sense. 
An immediate corollary is that intersection homology remains
invariant under such stratified formal deformations
(Corollary \ref{cor.ihstable}). These results can be
reinterpreted as giving conditions under which the intersection
homology of a given space can be computed directly from a
Vietoris-Rips type complex of data points near the space
(Corollary \ref{cor.ihfromvr}), but the conditions may be
hard to check a priori.

Whitehead's theory implies in particular the concept of a spine of
a complex. Intersection homology satisfies Poincar\'e duality
for spaces that are so-called pseudomanifolds. These allow for a notion of
orientability. It is an easy observation 
(see Lemma \ref{lem.pseudomfdisspine})
that if in a given complex
one looks for pseudomanifolds obtainable by simplicial collapses, 
then one must seek them among the
spines of that complex. Therefore, and for reasons mentioned earlier,
spines constitute a particularly
important class of subcomplexes. Our stratified simple homotopy 
transformations
lead to a notion of a \emph{stratified spine} of a stratified complex.
In general, a given complex will have nonhomeomorphic stratified spines, but
our main theorem implies that any two (codimension preserving)
stratified spines have isomorphic intersection homology groups.
Of course a given (ordinary, i.e. unstratified) expansion 
of a pseudomanifold $X$ may not possess a stratified spine which
is homeomorphic to $X$, or which even just has the same intersection
homology as $X$ 
(though it trivially always contains an ordinary spine with that
property, namely $X$ itself).
In such cases, different methods are required to determine the
intersection homology of $X$ from the expansion.

We have implemented the execution of stratified collapses,
computation of stratified spines, and the calculation of intersection
homology with $\intg/2\intg$-coefficients in Python, and use
this implementation to provide
various examples of stratified spines, and their intersection
homology, of Delaunay-Vietoris-Rips type complexes associated to
data points sampled near various singular spaces.
From the computational point of view, the idea of determining
a spine prior to computing homology has the advantage of
leading to smaller matrices representing chain complex
boundary operators.
Vietoris-Rips and similar complexes generally contain a vast number
of topologically insignificant simplices, and spines tend to reduce
that number substantially. In particular, the dimension of the spine
will often be lower than the dimension of the original complex.
This also means that in practice the ``true'' codimension of a singular
point can be better estimated from its geometric codimension in the spine
than from its geometric codimension in the original complex.

Interesting questions not treated in this paper concern sampling
methods and stratification learning.  
The question, for example, of describing conditions on point distributions
near a singular space $X$ that will ensure that the associated 
\v{C}ech- or Vietoris-Rips-type complex contains stratified spines
homeomorphic (say) to $X$ requires quite different methods.
Throughout the paper we assume, as do Bendich and Harer in \cite{bendichharer},
that for each point it is known whether to deem it regular or singular.
See however \cite{blrs} for some suggestions of heuristics.

The paper is organized as follows:
Section \ref{sec.simplcollexpspines} recalls basic material on 
simplicial complexes and
Whitehead's classical simple homotopy theory, particularly the
notions of simplicial collapses, expansions and spines.
Intersection homology is reviewed in 
Section \ref{sec.ih}. We define it in a rather general setting of
filtered spaces and do not limit ourselves to pseudomanifolds.
Perversity functions are also allowed to be more general than in
the classical papers of Goresky and MacPherson.
Section \ref{sec.stratspines} develops the central notions of 
stratified collapses
and expansions, as well as stratified spines. Several results are
proven that clarify the effect of commuting these
operations; these results are used in the algorithm design. 
To relate formal combinatorial collapses to
continuous deformations, we introduce in Section \ref{sec.freelyorthogdefretr}
particular representatives of stratified homotopy equivalences
called freely orthogonal deformation retractions.
The main theoretical results are contained in 
Section \ref{sec.formaldefhtpytypeih}, while
Section \ref{sec.implementation} discusses algorithm design 
and miscellaneous aspects
of our computer implementation. We conclude with exemplary executions 
of these algorithms on various sampled point clouds
in Section \ref{sec.examples}.\\

\textbf{Acknowledgements:}
We would like to thank Bastian Rieck for
early discussions on ordinary spines and computations using his
\texttt{Aleph}-package.

\section{Simplicial Collapses, Expansions and Spines}
\label{sec.simplcollexpspines}

We use the term \emph{simplicial complex} in the sense of \cite[\S 3, p. 15]{munkres};
see also \cite[2.27, p. 26]{rosa}.
Thus a simplicial complex is a set of finite nonempty sets such that if
$s\in K$, then every nonempty subset (``face'') of $s$ is also in $K$.
The elements $s \in K$ of a simplicial complex $K$ are referred to as its 
(closed) simplices. The dimension of a simplex $s$ is its cardinality minus $1$.
The $0$-dimensional simplices of $K$ will be called \emph{vertices}.
The set of vertices of $K$ will be denoted $K^0$.
Note that a simplex is thus uniquely determined by
its vertices.
For $s,t \in K$, it will be convenient to write
$t\leq s$ if $t$ is a face of $s$, and $t<s$ if $t$ is a proper
face of $s$.
The associated \emph{polyhedron} (or \emph{geometric realization})
of a simplicial complex $K$ is a topological space denoted by $|K|$.
Thus we are careful to distinguish between a simplicial complex $K$ and the 
topological space $|K|$. 
If we wish to distinguish denotationally between 
regarding a simplex $s$ as an element of $K$
or as a closed subset of $|K|$ (given by the convex hull of its vertices), 
then we shall write $|s|$ for the latter.
Points $x\in |K|$ can be described using \emph{barycentric coordinates}:
there are unique numbers $x_u \in [0,1]$ for every vertex $u\in K^0$ such that
\[ x= \sum_{u\in K^0} x_u u,~ \text{  and } \sum_{u\in K^0} x_u =1. \]
Note that our simplicial complexes are abstract
and do not come equipped with an embedding into some Euclidean space.
This is in line with the intended applications in topological data science,
where simplicial complexes such as the \v{C}ech complex or Vietoris-Rips
complex of a data point cloud are given as abstract complexes without
a preferred embedding in a Euclidean space.
\begin{defn}
A simplex in $K$ is called \emph{principal} (in $K$), if it is not
a proper face of any simplex in $K$.
\end{defn}
Note that if $K$ is finite-dimensional, then all top-dimensional
simplices are principal, but there may well be principal simplices
that are not top-dimensional. 
\begin{defn}
A simplex $s$ in $K$ is called \emph{free} (in $K$) if\\
(1) $s$ is a proper face of a principal simplex $p$ in $K$, and\\
(2) $s$ is not a proper face of any simplex in $K$ other than $p$. 
\end{defn}
Thus if $s$ is a free simplex, then it is the proper face of
precisely one simplex $p\in K$, and this $p$ must be principal.
We may hence put
\[ \Princ_K (s) := p. \]
If $K$ is understood, then we shall also simply write $\Princ (s)$.
\begin{lemma} \label{lem.freehascodim1}
If $s\in K$ is free, then
$\dim \Princ (s) - \dim s =1.$
\end{lemma}
\begin{proof}
Since $s<\Princ (s),$ we have the bound $\dim s \leq \dim \Princ (s)-1.$
Suppose that $\dim s < \dim \Princ (s)-1.$
Then there is a simplex $t\in K$ such that $s<t<\Princ (s)$.
That is, $s$ is a proper face of $t$ but $t\not= \Princ (s)$,
a contradiction to the freeness of $s$.
\end{proof}
To define the notion of an elementary collapse of simplices, let
$K$ be a simplicial complex and $s\in K$ a free simplex.
Then
\[ K' := K - \{ s, \Princ_K (s) \} \]
is again a simplicial complex, as it is still closed under the operation of 
taking faces: If $t$ is any simplex in $K'$ and $t_0$ is a face of $t$,
then $t_0$ is still in $K'$, for otherwise $t_0$ would have to be
equal to $s$ or $\Princ (s)$. If $t_0 =s$, then $s$ would be a proper face
of both $t$ and $\Princ (s)$ with $t\not= \Princ (s)$. This contradicts
the freeness of $s$ in $K$. On the other hand, if $t_0 = \Princ (s)$, 
then $t_0$ would be principal in $K$ and so could not be a proper face of $t$.
The following definition can be traced back at least to Whitehead's
paper \cite{whitehead}, though precursors date back even further.
\begin{defn}
We say that $K'$ has been obtained from $K$ by an
\emph{elementary collapse} (of $\Princ (s)$, using its free face $s$).
We also say that $K$ has been obtained from $K'$ by an 
\emph{elementary expansion}.
\end{defn}
If $K'$ has been obtained from $K$ by an elementary collapse using the
free simplex $s$, then we shall also write $K' = K_s$. If $s'$ is a 
free simplex in $K_s$, we shall also write $(K_s)_{s'}$ as $K_{s,s'}$.
The space $|K_s|$ is a deformation retract of the space $|K|$.
In particular, $|K_s|$ and $|K|$ are homotopy equivalent, but more is
true: $|K_s|$ and $|K|$ are in fact simple homotopy equivalent
in the sense of simple homotopy theory, \cite{cohen}.
\begin{defn}
We say that a simplicial complex $K$ 
\emph{collapses} (simplicially) to a subcomplex $L\subset K$,
if $L$ can be obtained from $K$ by a finite sequence of elementary collapses.
In this case we shall write $K\searrow L$.
The complex $L$ is then also said to \emph{expand} to $K$,
written $L\nearrow K$.
\end{defn}
If $K$ collapses to $L$, then
the space $|L|$ is a deformation retract of the space $|K|$, and
$|L|$ and $|K|$ have the same simple homotopy type.
By its very definition, the relation $\searrow$ is transitive:
if $K \searrow K'$ and $K' \searrow K''$, then $K\searrow K''$.
Due to their combinatorial nature, collapses can be carried
out algorithmically, see Section \ref{sec.implementation}.
\begin{defn}
A finite sequence
\[ K=K_0 \to K_1 \to \cdots \to K_m =L,  \]
where each arrow represents a simplicial expansion or collapse,
is called a \emph{formal deformation} from $K$ to $L$.
\end{defn}
\begin{defn}
A subcomplex $S\subset K$ of a simplicial complex $K$ is called a
\emph{spine} of $K$, if\\
(1) $K$ collapses to $S$, and\\
(2) $S$ does not possess a free simplex.
\end{defn}
If $S$ is a spine of $K$, then we shall also refer to the polyhedron
$|S|$ as a \emph{spine} of the polyhedron $|K|$.
Note that condition (2) is an absolute condition on $S$, independent of
the ambient complex $K$. So if $L \subset K$ is a subcomplex with
$K\searrow L$ and $S$ is a spine of $L$, then $S$ is also a spine of $K$. 
Furthermore, for such $K,L$, if $S$ is a spine of $K$ such that  $L\searrow S$,
then $S$ is a spine of $L$.
These observations allow for induction arguments involving spines.

\section{Filtered Spaces and Intersection Homology}
\label{sec.ih}

A filtered topological space has intersection homology groups.
Filtrations arise for example when a space has singular, i.e. non-manifold, points 
near which the space does not look like a Euclidean space. 
Stratification theory then tries to
organize the singular points, according to their local type, into strata which 
themselves are manifolds (of various dimensions). The filtration is obtained
by taking unions of such strata. Stratifications can generally be constructed
for polyhedra of simplicial complexes, orbit spaces of group actions on manifolds, 
real or complex algebraic varieties, as well as for many interesting compactifications
of noncompact spaces, e.g. various moduli spaces. A chief motivation for developing
intersection homology was to find a theory which satisfies a form of
Poincar\'e duality for singular spaces. 
This duality principle holds for the ordinary homology of 
(closed, oriented) manifolds,
but ceases to hold for ordinary homology in the presence of singularities.
Intersection homology was introduced by Goresky and MacPherson in
\cite{gm1} and \cite{gm2}; a comprehensive treatment can be found in
\cite{borel}.
Modern resources on intersection homology
theory include \cite{kirwanwoolf}, \cite{friedman} and \cite{banagltiss}. 

\begin{defn}
A \emph{filtered space} is a Hausdorff topological space $X$ together with a
filtration
\[ X = X_n \supset X_{n-1} \supset \cdots \supset X_0 \supset X_{-1} = \varnothing  \]
by closed subsets. We call $n$ the \emph{formal dimension} of $X$.
The connected components of the difference sets
$X_j - X_{j-1}$ are called the \emph{strata of formal dimension $j$} of the filtered space.
If $S$ is a stratum of formal dimension $j$, then $\codim S := n-j$ 
is its \emph{formal codimension} in $X$. Since all strata contained in $X_j - X_{j-1}$
have the same formal codimension $n-j$, we may call this number the formal codimension of 
$X_j - X_{j-1}$.
We write $\Xa$ for the set of strata. The strata contained in $X_n - X_{n-1}$
are called the \emph{regular strata}. The other strata are called \emph{singular}.
The union of the singular strata is called the \emph{singular set} of $X$ and is 
denoted by $\Sigma_X$.
\end{defn}
Note that a stratum is singular if and only if its formal codimension is positive.
In practice there are various ways to assign formal codimensions to
strata: Without further a priori knowledge or preprocessing, one may
simply wish to take the geometric codimension in a given polyhedron.
Better estimates may be obtained by first computing stratified spines
and then taking geometric codimension in the spine,
or by computing persistent local homology groups near singular points.

We recall the definition of the singular intersection homology groups 
$IH_i^{\bar{p}} (X)$ of a filtered space $X$. The basic idea is the following:
Suppose that $X$ is a filtered space for which global orientability can be defined,
for example, a so-called pseudomanifold.
If one defines a chain complex by allowing only
chains that are transverse to the strata and computes its $i$-th homology, 
then, by a theorem of McCrory (see e.g. \cite{mccrory}, Theorem 5.2), 
one obtains the cohomology $H^{n-i} (X)$ of $X,$ 
where $n$ is the dimension of $X$ and $X$ is assumed to satisfy
a certain normality condition. Hence,
if one could move every chain to be transverse to the stratification, then
Poincar\'e duality would hold if $X$ is orientable. 
However, as noted above, classical Poincar\'e duality fails for 
general oriented singular spaces,
thus so does transversality. The idea of Goresky and MacPherson was to introduce a parameter,
which they called a ``perversity,'' that specifies the allowable deviation 
from full
transversality, and to associate a group to each value of the parameter,
thereby obtaining a whole spectrum of groups ranging from cohomology
to homology. The following definition of perversity is somewhat more general
than the original definition in \cite{gm1}.
\begin{defn} \label{def.ihperversity}
A \emph{perversity} $\bar{p}$ 
is a function associating to each natural number
$k=0,1,2,\ldots$ an integer $\bar{p} (k)$ such that $\bar{p} (0) =0$.
\end{defn}
The natural number $k$ is to be thought of as the codimension of a stratum.
Let $X$ be a filtered space and $\bar{p}$ a perversity.
Recall that a \emph{singular $i$-simplex} $\sigma$ in $X$ is a continuous
map $\sigma: \Delta^i \to X$, where $\Delta^i$ denotes the standard $i$-simplex.
The degree-$i$ singular chain group $C_i (X)$ of $X$ is freely generated by the
singular $i$-simplices in $X$. Using the alternate sum of the codimension $1$ faces
of a singular simplex, one obtains a boundary map $\partial_i: C_i (X)\to C_{i-1} (X)$.
Together with these boundary maps, the $C_i (X),$ $i=0,1,2,\ldots,$ form a chain
complex $C_* (X),$ the singular chain complex of $X$. Its homology is ordinary
homology $H_* (X)$. 
\begin{defn}
A singular $i$-simplex $\sigma: \Delta^i \to X$ is called
\emph{$\bar{p}$-allowable} if for every stratum $S\in \Xa$, the preimage
$\sigma^{-1} (S)$ is contained in the union of all
$(i-\codim (S) + \bar{p}(\codim S))$-dimensional faces of $\Delta^i$.
A chain $\xi \in C_i (X)$ is called \emph{$\bar{p}$-allowable}, if
every singular simplex $\sigma$ appearing with nonzero coefficient in $\xi$
is $\bar{p}$-allowable.
\end{defn}

\begin{defn}
The \emph{group $IC_i^{\bar{p}} (X)$ of $i$-dimensional (singular) intersection
chains of perversity $\bar{p}$} is defined to be the subgroup of $C_i (X)$
given by all chains $\xi$ such that $\xi$ and $\partial_i \xi$ is $\bar{p}$-allowable.
\end{defn}
The boundary map $\partial_i$ on $C_i (X)$ restricts to a map
$\partial_i : IC_i^{\bar{p}} (X)\longrightarrow 
  IC_{i-1}^{\bar{p}} (X)$
due to the condition imposed on $\partial_i \xi.$ Thus
$\{ (IC_i^{\bar{p}} (X), \partial_i ) \}$ forms a subcomplex
$IC_*^{\bar{p}} (X)$ of the singular chain complex $C_* (X).$
\begin{defn}
The homology groups of the singular intersection chain complex,
\[ IH_i^{\bar{p}} (X) = H_i ( IC_*^{\bar{p}} (X)), \]
are called the \emph{perversity $\bar{p}$ (singular) intersection homology groups} of
the filtered space $X.$
\end{defn}
The groups thus defined have $\intg$-coefficients,
$IH_i^{\bar{p}} (X) = IH_i^{\bar{p}} (X;\intg).$
Groups $IH_i^{\bar{p}} (X;G)$ with coefficients in any abelian
group $G$ can also be defined in a straightforward manner.
For the implementation of our algorithms, we chose to work
with $G=\intg/2\intg$, as did Bendich and Harer \cite{bendichharer}.

Contrary to ordinary homology, intersection homology is
not invariant under general homotopy equivalences.
However, if one places suitable filtration conditions on a homotopy equivalence,
the intersection homology groups do become invariant under such equivalences.
Let us look at such conditions in more detail.
We shall from now on, for simplicity of exposition and because
this context already illustrates all important scientific issues, restrict our
attention to filtered spaces of the form
\begin{equation} \label{equ.twostrataspaces}
X=X_n \supset X_{n-1} = \cdots =X_{n-k+1} =
X_{n-k} = \Sigma_X \supset X_{n-k-1} = \cdots = X_{-1} = \varnothing, 
\end{equation}
which we will briefly denote as pairs $(X,\Sigma_X)$.
(Thus $\Sigma_X$ has formal codimension $k$.)
The methods introduced in the present paper can, without major difficulty,
be extended to more general filtrations.
\begin{defn}
Let $(X,\Sigma_X)$ and $(Y,\Sigma_Y)$ be filtered spaces.
A \emph{stratified map} $f: (X,\Sigma_X) \to (Y,\Sigma_Y)$
is a continuous map $f:X\to Y$ such that $f(\Sigma_X) \subset \Sigma_Y$
and $f(X-\Sigma_X) \subset Y-\Sigma_Y$. Such a map is called
\emph{codimension-preserving} if
$\codim \Sigma_X = \codim \Sigma_Y.$
\end{defn}
Let $I = [0,1]$ denote the compact unit interval.
The \emph{cylinder} on a filtered space $(X,\Sigma_X)$ is the filtered
space $(X\times I, \Sigma_X \times I)$, i.e.
$\Sigma_{X\times I} = \Sigma_X \times I$.
We shall also write $(X,\Sigma_X)\times I$ for the pair
$(X\times I, \Sigma_X \times I).$
If $n$ is the formal dimension of $X$, then the formal dimension of $X\times I$
is defined to be $n+1$ and we set $\codim (\Sigma_X \times I) := \codim \Sigma_X$.
\begin{defn}
A \emph{stratified homotopy} is a stratified map 
$H: (X\times I, \Sigma_X \times I) \to (Y, \Sigma_Y)$.
\end{defn}
Thus a stratified homotopy maps
$H(\Sigma_X \times I) \subset \Sigma_Y$
and $H((X-\Sigma_X) \times I) \subset Y-\Sigma_Y$; in particular,
the tracks of the homotopy remain within the stratum where they start.
\begin{defn}
Stratified maps $f,g: (X,\Sigma_X) \to (Y,\Sigma_Y)$ are
\emph{stratified homotopic} if there exists a stratified homotopy
$H: (X\times I, \Sigma_X \times I) \to (Y, \Sigma_Y)$ between $f$ and $g$, i.e.
$H_0 =f$ and $H_1 =g$.
\end{defn}
Note that if $f$ (or $g$) is in addition codimension-preserving, then
a stratified homotopy $H$ between $f$ and $g$ is automatically codimension-preserving
as well.

\begin{defn}
A stratified codimension-preserving map 
$f: (X,\Sigma_X) \to (Y,\Sigma_Y)$ is called a
\emph{stratified homotopy equivalence}, if there exists a
stratified (and then necessarily codimension-preserving) map 
$g: (Y,\Sigma_Y) \to (X,\Sigma_X)$ such that $g\circ f$ and $f\circ g$
are stratified homotopic to the identity.
\end{defn}
It is easy to show that singular intersection homology is invariant
under stratified homotopy equivalences, see e.g.
Friedman \cite[Prop. 4.1.10, p. 135; Cor. 6.3.8, p. 271]{friedman}.
\begin{prop} \label{prop.ihstratheinv}
A stratified homotopy equivalence $f: (X,\Sigma_X) \to (Y,\Sigma_Y)$ 
of filtered spaces
induces isomorphisms $f_*: IH^{\bar{p}}_i (X) \cong IH^{\bar{p}}_i (Y)$ for all
$i$ and all $\bar{p}$.
\end{prop}
Taking filtered spaces as objects and stratified homotopy equivalences
as morphisms yields a category:
\begin{lemma} \label{lem.compstrathes}
The composition of stratified homotopy equivalences is a stratified 
homotopy equivalence.
\end{lemma}
\begin{proof}
Let $f:(X,\Sigma_X) \to (Y,\Sigma_Y)$ and
$f': (Y,\Sigma_Y) \to (Z,\Sigma_Z)$ be stratified homotopy
equivalences with stratified homotopy inverses
$g:(Y,\Sigma_Y) \to (X,\Sigma_X)$ and
$g': (Z,\Sigma_Z) \to (Y,\Sigma_Y)$.
Since $f$ and $f'$ are codimension-preserving, the formal codimensions
of the singular sets coincide,
\[ \codim_X \Sigma_X = \codim_Y \Sigma_Y = \codim_Z \Sigma_Z. \]
Let
$H: (X,\Sigma_X)\times I \to (X,\Sigma_X)$ be a
stratified homotopy between
$H_0 = g \circ f$ and $H_1 = \id_X$,
let
$H': (Y,\Sigma_Y)\times I \to (Y,\Sigma_Y)$ be a
stratified homotopy between
$H'_0 = g' \circ f'$ and $H'_1 = \id_Y$.
As $H,H'$ are stratified, the singular sets and their complements
are mapped compatibly,
\[ H(\Sigma_X \times I)\subset \Sigma_X,~
   H((X-\Sigma_X) \times I)\subset X-\Sigma_X, \]
\[ H'(\Sigma_Y \times I)\subset \Sigma_Y,~
   H'((Y-\Sigma_Y) \times I)\subset Y-\Sigma_Y. \]   
Recall that the \emph{concatenation} $F*F'$ of two homotopies
$F,F': A\times I \to B$ 
such that $F_1 = F'_0$
is the homotopy $F*F': A\times I \to B$
given by
\[ (F*F')(a,t) = \begin{cases}
F(a,2t),& t\in [0,\smlhf], \\
F'(a, 2t-1),& t\in [\smlhf, 1].
\end{cases} \]
Note that if $A' \subset A$, $B' \subset B$ are subspaces, then
$(F*F') (A' \times I) \subset B'$ if and only if
$F(A' \times I)\subset B'$ and $F' (A' \times I)\subset B'$.
Now let $F: X\times I \to X$ be the homotopy given by the composition
\[ X\times I \stackrel{f\times \id}{\longrightarrow} Y\times I
  \stackrel{H'}{\longrightarrow} Y
  \stackrel{g}{\longrightarrow} X. \]
Since
\[ F_1 = g \circ H'_1 \circ (f\times 1) 
   = g \circ  f = H_0, \]
we may concatenate to a homotopy
$G := F*H: X\times I \longrightarrow X.$
This is then a homotopy between $G_1 = H_1 = \id_X$ and
\[ G_0 = F_0 = g \circ H'_0 \circ (f\times 0) =
   gg'f'f. \]
It satisfies $G(\Sigma_X \times I) \subset \Sigma_X,$ since
\[ F(\Sigma_X \times I) =
  g(H'((f\times \id)(\Sigma_X \times I)))
    \subset g(H'(\Sigma_Y \times I)) \subset g(\Sigma_Y)
      \subset \Sigma_X \]
and $H(\Sigma_X \times I) \subset \Sigma_X$.
For the complements,
\begin{align*}
F((X-\Sigma_X) \times I) 
&= g(H'((f\times \id)((X-\Sigma_X) \times I)))
    \subset g(H'((Y-\Sigma_Y) \times I)) \\
&\subset g(Y-\Sigma_Y) 
      \subset X-\Sigma_X 
\end{align*}      
and $H((X-\Sigma_X) \times I) \subset X-\Sigma_X$.      
Consequently, $G((X-\Sigma_X)\times I)\subset X-\Sigma_X$
and $G$ is a stratified homotopy between
$(gg')\circ (f'f)$ and $\id_X$.
Note that $f'f$ and $gg'$ are codimension-preserving.
By symmetry, a stratified homotopy
$G':(Z,\Sigma_Z)\times I \to (Z,\Sigma_Z)$ between 
$(f'f) \circ (gg')$ and
$\id_Z$ is obtained by the same method applied
to stratified homotopies
$K: (Y,\Sigma_Y)\times I \to (Y,\Sigma_Y)$ between
$K_0 = f \circ g$ and $K_1 = \id_Y$, and
$K': (Z,\Sigma_Z)\times I \to (Z,\Sigma_Z)$ between
$K'_0 = f' \circ g'$ and $K'_1 = \id_Z$.
\end{proof}

While singular intersection homology is not directly amenable to algorithmic computation,
a simplicial version of intersection homology is. It must then be clarified under which
conditions these versions are isomorphic. We will review this briefly and ask the
reader to consult e.g. Friedman's monograph \cite{friedman} for details.

Suppose that the filtered space $X$ can be triangulated such that the filtration
subspaces are triangulated by simplicial subcomplexes.
Let $T: |K| \cong X$ be a choice of such a triangulation. The closed subspaces
$X_j \subset X$ are given by $X_j = T(|K_j|)$ for subcomplexes $K_j \subset K$.
One defines the simplicial intersection chain complex
$IC^{\bar{p}, T}_* (X)$ as a subcomplex of the 
simplicial chain complex of $K$, using $T$ and simplices of $K$ that are $\bar{p}$-allowable
with respect to the filtration $\{ K_j \}$. The homology of $IC^{\bar{p}, T}_* (X)$
is the \emph{simplicial intersection homology} $IH^{\bar{p}, T}_* (X)$ associated to
the triangulation $T$.
This is algorithmically computable,
see Section \ref{sec.implementation} and \cite{bendichharer}.

The dependence on choices of triangulations can be controlled as follows:
A piecewise-linear (PL) space is a second-countable Hausdorff space $X$ together with a collection
$\Ta$ of locally finite triangulations, closed under simplicial
subdivision, and such that any two triangulations in $\Ta$ have a common subdivision.
Suppose that the PL space $X$ is \emph{PL filtered}, that is, equipped with a filtration by
closed subsets $X_j$ that are of the form $T(|K_j|)$ for some triangulation $T$ in $\Ta$.
(One then calls $T$ \emph{compatible} with the filtration.)
Taking the direct limit of the simplicial groups $IH^{\bar{p}, T}_* (X)$ over the
directed set $\Ta$ defines the \emph{PL intersection homology}
$IH^{\bar{p}, \PL}_* (X)$ of a PL filtered space $X$.

When does a particular simplicial intersection homology group compute PL intersection homology?
A simplicial subcomplex $L\subset K$ is called \emph{full} if membership of simplices in $L$
can be recognized on the vertex level, i.e.: whenever a simplex $s\in K$ has all of its
vertices in $L$, then $s$ itself must be in $L$.
A compatible triangulation $T\in \Ta$ of a PL filtered space $X$ is said to be \emph{full}
if $T$ triangulates the $X_j$ as full subcomplexes.
Using finitely many subdivisions of a compatible triangulation, one sees that a PL
filtered space always possesses a full triangulation.
The condition of fullness can of course be checked algorithmically.
\begin{prop} \label{prop.simplihcomputesplih}
(Goresky, MacPherson \cite[Appendix]{mv}; see also \cite[Thm. 3.3.20, p. 119]{friedman})
If $T$ is a full (compatible) triangulation of a PL filtered space $X$, then
the canonical map $IH^{\bar{p}, T}_* (X) \to IH^{\bar{p}, \PL}_* (X)$
is an isomorphism.
\end{prop}
It remains to relate PL intersection homology to singular intersection homology.
For a compact PL space $L$, let
$c^\circ L$ denote the open cone $([0,1)\times L)/(0\times L)$.
A PL filtered space $X$ is called a \emph{PL CS set} if it is locally cone-like and 
all filtration differences $X_j - X_{j-1}$ are $j$-dimensional PL manifolds.
Points $x\in X_j - X_{j-1}$ must thus have open neighborhoods that are PL
homeomorphic to $\real^j \times c^\circ L$ under a filtration preserving
homeomorphism, where $L$ is some compact PL filtered space.
Note that this concept creates in particular a logical connection
between the topological manifold dimension of strata and the formal
dimension of strata in the filtered space. 
Given a finite dimensional PL space $X$, there always exists
a PL filtration of $X$ such that $X$ becomes a PL CS set with respect to
this filtration. A given PL filtration may well not satisfy the PL CS
condition.
\begin{example}
Let $X=|K|$ be the polyhedron of the simplicial complex $K$ generated
by the $1$-simplices $\{ v_0, v_1 \},$
$\{ v_0, v_2 \}$ and $\{ v_0, v_3 \}.$
Let $\Sigma_X \subset X$ be the closed PL subspace given by
the polyhedron of the complex generated by $\{ v_0, v_1 \},$
$\{ v_0, v_2 \}$. Then $\Sigma_X$ is PL homeomorphic to a compact
interval and hence a $1$-dimensional PL manifold (with boundary).
The pair $(X,\Sigma_X)$ does not yet constitute a PL filtered space because
no formal dimensions have been assigned.
Can we make such an assigment in a way that will make the resulting
PL filtered space into a PL CS set?
If so, then we need to assign formal dimension $1$ to $\Sigma_X$, i.e.
we are forced to set $X_1 := \Sigma_X$. Since $\Sigma_X$ is a proper
subset of $X$, the formal dimension of $X$ would have to be at least $2$
(and thus would not agree with the polyhedral dimension of $X$).
If such a filtration made $(X,\Sigma_X)$ into a PL CS set, then
the point $v_0 \in \Sigma_X$ would have a neighborhood PL homeomorphic
to a space of the form $\real^1 \times c^\circ L$ for some compact 
PL space $L$. No such $L$ exists. (For dimensional reasons, $L$ would
have to be empty, but the neighborhood of $v_0$ is not homeomorphic
to $\real^1$.) We conclude that the PL filtered space $(X,\Sigma_X)$
is not a PL CS set. If instead one took
$\Sigma_X = X_0 = \{ v_0, \ldots, v_3 \}$ and $X_1 = X$, then
$(X,\Sigma_X)$ would be a PL CS set. 
\end{example} 

\begin{prop} \label{prop.plihcomputessingih}
(Friedman \cite[Thm. 5.4.2, p. 229]{friedman}.)
If $X$ is a PL CS set, then there is an isomorphism
$IH^{\bar{p}, \PL}_* (X) \cong IH^{\bar{p}}_* (X)$
between PL and singular intersection homology groups.
\end{prop}
Together, Propositions \ref{prop.simplihcomputesplih} 
and \ref{prop.plihcomputessingih}
provide sufficient conditions for simplicial intersection homology
to compute singular intersection homology.
The upshot is that every PL space $X$ 
has a PL filtration $\{ X_j \}$ which makes $X$ into a PL CS set,
and a triangulation
whose simplicial intersection homology computes the singular intersection
homology of $(X, \{ X_j \})$.

A particularly important class of singular spaces are pseudomanifolds.
An \emph{$n$-dimensional PL pseudomanifold}
is a polyhedron $X$ for which some
(and hence every) triangulation has the following
property:
Every simplex is the face of some $n$-simplex, and
every $(n-1)$-simplex is a face of exactly two $n$-simplices.
Pseudomanifolds admit a concept of orientability, and if a 
given pseudomanifold is oriented, its intersection homology
will satisfy a generalized form of Poincar\'e duality.

The following small example shows that direct computation of intersection
homology from the Vietoris-Rips complex of a point cloud does not
usually yield the correct intersection homology of an underlying
space near which the points are sampled, even when the points have
been sampled ``well'' in the sense that the location of the
singular set is known and the Vietoris-Rips complex
is homotopy equivalent to the underlying space (and hence its ordinary
homology is correct).
\begin{example} \label{expl.s1vs112pts}
Let $\bar{0}$ denote the zero-perversity whose value is $0$ on every
codimension, and let $\overline{-1}$ the perversity whose value is $-1$
on every positive codimension. We shall write $\intg_2$ for 
$\intg/2\intg$.
The figure eight space $S^1 \vee S^1$, a wedge of two circles,
equipped with the obvious filtration with one singular point in
codimension $1$, has intersection homology
\begin{align*}
IH^{\bar{0}}_{i} (S^1 \vee S^1;\intg_2) &= 
\begin{cases}
\intg_2^2, &\mbox{ if } i= 1\\
\intg_2, &\mbox{ if } i= 0\\
0, &\mbox{ else;}
\end{cases}
\\
IH^{\overline{-1}}_{i} (S^1 \vee S^1;\intg_2) &=
\begin{cases}
\intg_2^2, &\mbox{ if } i= 0\\
0, &\mbox{ else.}
\end{cases}
\\
\end{align*}
Figure \ref{fig.S1vS1VR12pts} shows 12 points near an embedding of
$S^1 \vee S^1$ in the plane $\real^2$, as well as 
an associated Vietoris-Rips complex with polyhedron $X$.
\begin{figure}
	\includegraphics[width=0.9\textwidth]{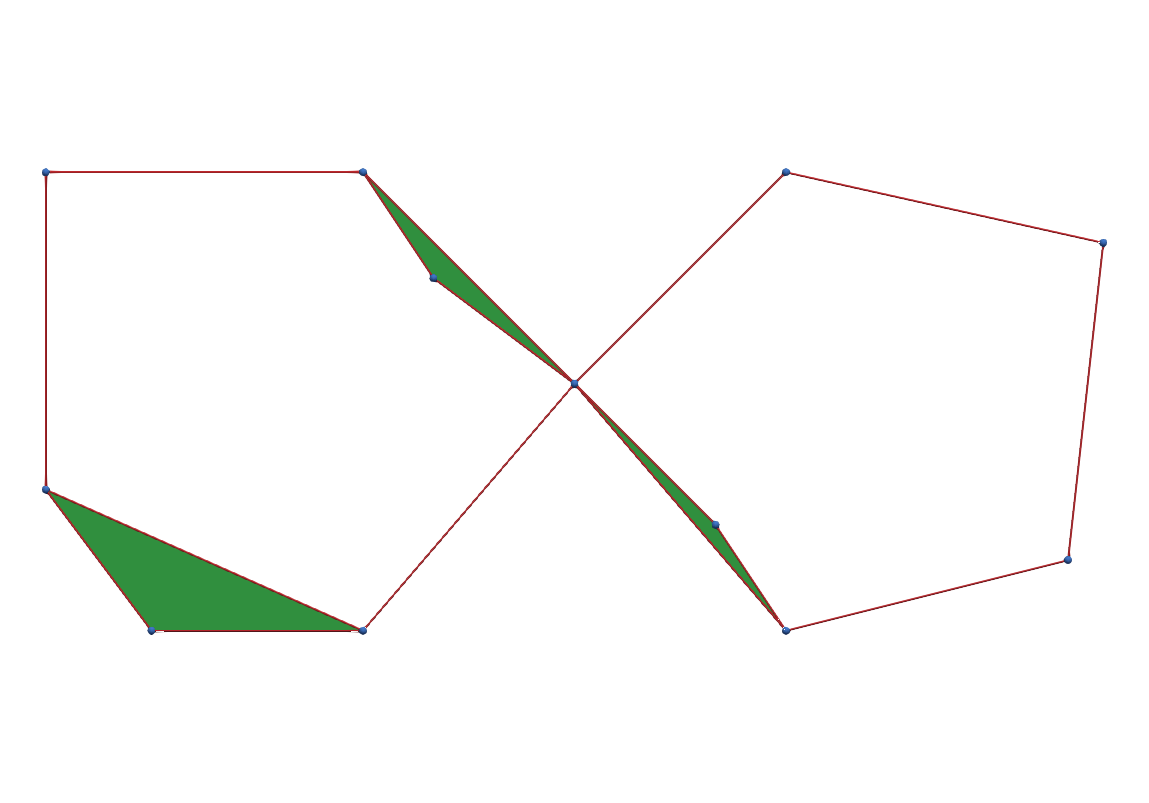}
	\caption{Vietoris-Rips polyhedron $X$ of point cloud near $S^1 \vee S^1$.} 
	\label{fig.S1vS1VR12pts}
\end{figure}
The singular point $s$ of $S^1 \vee S^1$ is one of the 12 points
and is located in the middle of the figure.
The polyhedron $X$ is $2$-dimensional and 
filtered by
$\{ s \} =X_0 \subset X_2 =X$. Its intersection homology groups are
\[
IH^{\bar{0}}_{i} (X;\intg_2) =
IH^{\overline{-1}}_{i} (X;\intg_2) =
\begin{cases}
\intg_2^2, &\mbox{ if } i= 0\\
0, &\mbox{ else,}
\end{cases}
\]
which are not isomorphic to the above groups of $S^1 \vee S^1$.
A spine $X'$ of $X$ is shown in Figure \ref{fig.S1vS1VR12ptsSpine}.
\begin{figure}
	\includegraphics[width=0.9\textwidth]{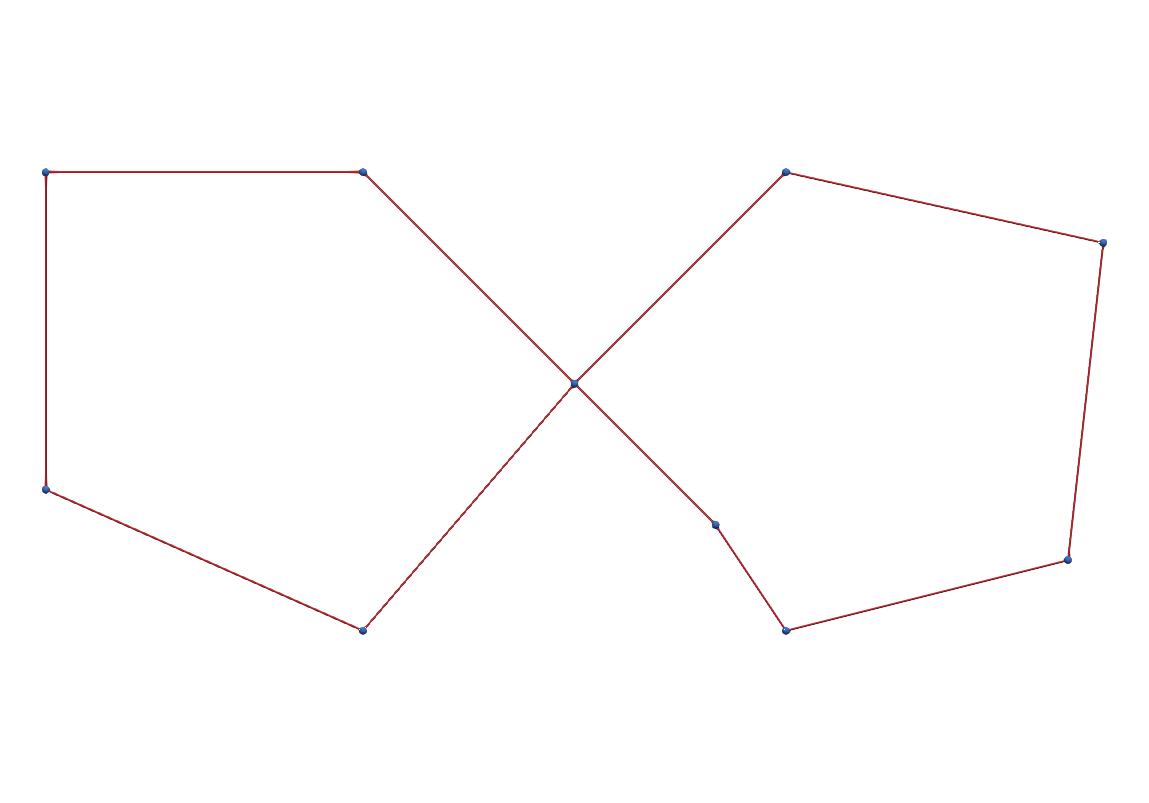}
	\caption{Spine $X'$ of Vietoris-Rips polyhedron $X$.} 
	\label{fig.S1vS1VR12ptsSpine}
\end{figure}
It is $1$-dimensional and filtered by 
$\{ s \} =X'_0 \subset X'_1 =X'$. A computer calculation of its
intersection homology yields
\begin{align*}
IH^{\bar{0}}_{i} (X';\intg_2) &= 
\begin{cases}
\intg_2^2, &\mbox{ if } i= 1\\
\intg_2, &\mbox{ if } i= 0\\
0, &\mbox{ else;}
\end{cases}
\\
IH^{\overline{-1}}_{i} (X';\intg_2) &=
\begin{cases}
\intg_2^2, &\mbox{ if } i= 0\\
0, &\mbox{ else,}
\end{cases}
\\
\end{align*}
which do agree with the intersection homology of $S^1 \vee S^1$.
Of course, one may readily deduce this from the fact that in this
particular example it so happens that the spine is even
stratum preserving homeomorphic to $S^1 \vee S^1$, which obviously does not
always happen. (We will return to this example in 
Section \ref{sec.examples}).

Such examples suggest the idea of using simplicial collapse processes
and spines to explore the intersection homology
of Vietoris-Rips and other simplicial complexes.
However, contrary to ordinary homology, intersection homology is
not generally invariant under arbitrary simplicial collapses,
so a second new idea is needed, to be developed in the next section:
the idea of \emph{stratified} simplicial collapses and expansions.
\end{example}

\section{Stratified Formal Deformations and Stratified Spines}
\label{sec.stratspines}

We define stratified (simplicial) collapses and expansions,
stratified formal deformations and stratified spines.
We begin by observing that if in a simplicial complex 
one looks for pseudomanifolds obtainable by collapses
(i.e. the complex is a thickening of an unknown pseudomanifold), then
one must seek these pseudomanifolds among the spines of that complex.
\begin{lemma} \label{lem.pseudomfdisspine}
Let $K$ be a simplicial complex and $L\subset K$ a subcomplex.
If $K\searrow L$ and $|L|$ is a pseudomanifold, then $L$ is a spine of $K$.
\end{lemma}
\begin{proof}
Let $n$ be the dimension of $|L|$.
Suppose that $L$ is not a spine of $K$.
Then $L$ possesses a free simplex $s$. Let $p = \Princ_L (s)$ be the associated
principal simplex. We have $\dim p \leq n$.
If $\dim p=n$, then $\dim s= n-1$ by Lemma \ref{lem.freehascodim1}. 
So $s$ is an $(n-1)$-simplex
which is the face of precisely one $n$-simplex, contradicting the fact that in
an $n$-dimensional pseudomanifold, every $(n-1)$-simplex is the face of precisely
two $n$-simplices. Suppose that $\dim p<n$.
Then $p$, not being the face of any simplex in $L$, is in particular not the face of any
$n$-simplex, contradicting the fact that in an $n$-dimensional pseudomanifold, every simplex
must be the face of some $n$-simplex.
\end{proof}

\begin{defn}
A \emph{layered simplicial complex} is a triple $(K,C,S)$, where 
$K$ is a simplicial complex and $C,S$ are disjoint subcomplexes of $K$.
Simplices of $K$ which are neither in $C$ nor in $S$ are called
\emph{intermediate simplices}. We shall write 
$\IM (K,C,S)$ for the set of intermediate simplices.
\end{defn}
Note that in general $\IM (K,C,S)$ is merely a subset, not a subcomplex, of $K$.

\begin{defn}
A \emph{divided simplicial complex} is a pair $(K,S^0)$, where
$K$ is a simplicial complex and $S^0 \subset K^0$ is a set of vertices of $K$.
\end{defn}

A divided simplicial complex $(K,S^0)$ gives rise to a layered simplicial complex
$(K,C,S)$ such that the set of vertices of $S$ is $S^0$, as follows:
Take $S$ to consist of all simplices of $K$ whose vertices lie in $S^0$.
(Note that this is indeed a subcomplex of $K$.)
Take $C$ to consist of all simplices of $K$ whose vertices lie in the complement $C^0 := K^0 - S^0$.
(Again, this is a subcomplex of $K$.)
We call $(K,C,S)$ the \emph{associated} layered complex of the divided complex $(K,S^0)$.
In the associated layered complex, the intermediate simplices are those simplices of $K$
that have at least one vertex in $S^0$ and one vertex in $K^0 - S^0$. Hence the dimension
of an intermediate simplex in an associated layered complex is at least $1$.

Let $(K,C,S)$ be a layered simplicial complex.
Suppose that 
\[ K =K_0 \searrow K_1 \searrow K_2 \searrow \cdots \searrow K_m = K_S \]
is a finite sequence of elementary collapses, such that 
$K_{i+1} = K_i - \{ s, p \},$
where $p$ is a simplex in $S$ which is principal in $K_i$, and
$s$ is a face of $p$ which is free in $K_i$. Note that since $S$ is a subcomplex,
$s$ is in $S$. Since $C$ and $S$ are disjoint, neither $p$ nor $s$
can be a simplex of $C$. Hence $C$ remains a subcomplex of $K_i$ for every $i$,
and in particular is a subcomplex of $K_S$.
Set 
\begin{equation} \label{equ.sikicaps}
S_i = K_i \cap S. 
\end{equation}
The intersection $A\cap B$ of two subcomplexes $A,B\subset L$ of a simplicial
complex $L$ is a subcomplex of $A$, of $B$, and of $L$.
Hence $S_{i+1}$ is a subcomplex of $S_i$ for every $i$.
The simplices $p$ and $s$ lie in the complex $S_i$.
As $p$ is principal in $K_i$, it is in particular principal in the subcomplex $S_i$.
Furthermore, as $s$ is free in $K_i$, it is also free in the subcomplex $S_i$.
This shows that $S_{i+1}$ is obtained from $S_i$ by an elementary collapse $S_i \searrow S_{i+1}$.
Hence there is a finite sequence of elementary collapses
\[ S =S_0 \searrow S_1 \searrow S_2 \searrow \cdots \searrow S_m =: S'. \]
Since $S'$ is obtained by removing simplices of $S$, the subcomplexes $C$ and $S'$ of
$K_S$ are of course still disjoint and hence $(K_S, C, S')$ is a layered complex.
\begin{defn}
An \emph{$S$-collapse} of the layered simplicial complex $(K,C,S)$ is 
any layered simplicial complex $(K_S, C, S')$ obtained by the collapse process
described above. If $m=1$, we shall refer to $K_0 \searrow K_1$ also as an
\emph{elementary $S$-collapse}. 
\end{defn}
Loosely, an $S$-collapse is thus obtained by collapsing only simplices of the subcomplex $S$.
Since the definition of a layered complex is symmetric in $C$ and $S$,
the notion of a $C$-collapse can be defined by interchanging the role of $C$ and $S$
in the above collapse process. Thus:
\begin{defn}
A \emph{$C$-collapse} of the layered simplicial complex $(K,C,S)$ is 
any layered simplicial complex $(K_C, C', S)$ obtained by the collapse process
described above, where only simplices of $C$ are eligible for elementary collapses.
If $m=1$, we shall speak of an \emph{elementary $C$-collapse}. 
\end{defn}
Now, as $C$ has been left intact in any $S$-collapse $(K_S, C, S')$, 
\[ K =K_0 \searrow K_1 \searrow K_2 \searrow \cdots \searrow K_m = K_S, \]
\[ K_{i+1} = K_i - \{ s_i, p_i \}, \]
we may
execute a $C$-collapse 
\[ K_S =(K_S)_0 \searrow (K_S)_1 \searrow (K_S)_2 \searrow \cdots \searrow (K_S)_r = (K_S)_C, \]
\[ (K_S)_{i+1} = (K_S)_i - \{ s'_i, p'_i \}, \]
on $(K_S, C, S')$, yielding a layered simplicial complex
$((K_S)_C, C', S')$. 
Is $p'_0$ principal in $K$?
The simplex $p'_0$ lies in $C$ and is principal in 
\[ K_S = K - \{ s_0, p_0, \ldots, s_{m-1}, p_{m-1} \}. \]
So $p'_0$ is not the proper face of any simplex in $K_S$.
Hence, if $p'_0$ were not principal in $K$, then $p'_0$ would have to
be the face of some $p_i$. But this would place $p'_0$ in the complex $S$ which is
impossible as $S$ and $C$ are disjoint. 
We conclude that $p'_0$ is principal even in $K$.
For the same reason, $s'_0$ is free even in $K$:
We know $s'_0$ lies in $C$ and is free in $K_S$.
So $s'_0$ is not the proper face of any simplex in $K_S$ other than $p'_0$.
If $s'_0$ were not free in $K$, then $s'_0$ would be the proper face of some
simplex $t$ in $K$, $t\not= p'_0$. Then $s'_0$ would be the proper face of some $p_i$,
which is again impossible.
Thus the elementary collapse
\[ K = K'_0 \searrow K'_1 := K'_0 - \{ s'_0, p'_0 \} \]
is defined.
Similarly, one then sees by an easy induction argument
that $p'_i$ is principal in $K'_i$ and $s'_i$ is free in $K'_i$, whence we may define
\[ K'_{i+1} = K'_i - \{ s'_i, p'_i \}. \]
We thus obtain a $C$-collapse $(K_C, C'', S)$,
\[ K =K'_0 \searrow K'_1 \searrow K'_2 \searrow \cdots \searrow K'_r = K_C. \]

It is immediately clear that $p_0$ is principal in $K_C$, since it is so even in $K$.
For the same reason, $s_0$ is free in $K_C$.
Thus the elementary collapse
\[ K_C = (K_C)_0 \searrow (K_C)_1 := (K_C)_0 - \{ s_0, p_0 \} \]
is defined.
Similarly, one then sees by an easy induction argument
that $p_i$ is principal in $(K_C)_i$ and $s_i$ is free in $(K_C)_i$, whence we may define
\[ (K_C)_{i+1} = (K_C)_i - \{ s_i, p_i \}. \]
We thus obtain an $S$-collapse $((K_C)_S, C'', S'')$,
\[ K_C = (K_C)_0 \searrow (K_C)_1 \searrow (K_C)_2 \searrow \cdots \searrow (K_C)_m = (K_C)_S. \]
Now note that
\[ (K_C)_S = (K_S)_C,~ C'' = C',~ S'' = S'. \]
Thus $S$-collapses and $C$-collapses are commuting operations on layered simplicial 
complexes.

Notably with a view towards algorithm design, it is important to know
that $S$- and $C$-collapses do not affect the set of intermediate
simplices:
\begin{lemma} \label{lem.impresundercscoll}
Let $(K,C,S)$ be the associated layered complex of a divided complex.
If $(K',C,S')$ has been obtained from $(K,C,S)$ by an $S$-collapse, then
\[ \IM (K',C,S') = \IM (K,C,S). \]
If $(K',C',S)$ has been obtained from $(K,C,S)$ by a $C$-collapse, then
\[ \IM (K',C',S) = \IM (K,C,S). \]
\end{lemma}
\begin{proof}
By symmetry, it suffices to consider $S$-collapses.
Then $K' = K - \{ s,p \},$ $S' = S - \{ s,p \},$ with
$p\in S,$ $p$ principal in $K$, $s$ a proper face of $p$ and $s$
free in $K$.
The inclusion 
\[ \IM (K',C,S') \subset \IM (K,C,S) \] 
follows from
$K' \subset K$ and $S' \subset S$.
Conversely, suppose that $t\in \IM (K,C,S)$.
Then $t\not= s$ and $t\not= p$, since $t\not\in S$, while $s,p\in S$.
Therefore, $t\in K - \{ s,p \} = K'$.
Since $t$ is intermediate in $(K,C,S)$
and $(K,C,S)$ comes from a divided complex, $t$ has a vertex
$v\in C$ and a vertex $w\in S$. In particular, $t$ has positive dimension.
We claim that $w\in S'$:
As $p$ has positive dimension, $w \not= p$.
If $w$ were equal to $s$, then $s$ would be a proper
face of both $t$ and $p$, with $t\not= p$.
This would contradict the freeness of $s$ in $K$.
We deduce that $w \not= s$, so that $w\in S - \{ s,p \}$.
This proves the claim.
Thus $t$ is a simplex of $K'$ which possesses a vertex $v\in C$
and a vertex $w\in S'$, showing that $t$ is intermediate in
$(K',C,S')$.
\end{proof}

The most interesting aspect of stratified collapses is the treatment
of intermediate simplices.
Let $(K,C,S)$ be a layered simplicial complex and suppose that
\[ K =K_0 \searrow K_1 = K_0 - \{ s, p \} \]
is an elementary collapse, where 
\begin{enumerate}
\item[(i)] $p$ is an intermediate simplex of $(K,C,S)$,
\item[(ii)] $p$ is principal in $K_0$,
\item[(iii)] $s$ is a face of $p$ which is free in $K_0$, and
\item[(iv)] if $t$ is a simplex of $S$ which is a proper face of $p$, then
  $t$ is a proper face of $s$.
\end{enumerate}
Put $C_1 := K_1 \cap C$ and $S_1 := K_1 \cap S$.
Then $C_1$ and $S_1$ are subcomplexes of $K_1$ and they are disjoint.
Consequently, $(K_1, C_1, S_1)$ is a layered simplicial complex.
\begin{lemma} \label{lem.s1iss}
For any layered simplicial complex and simplices $s,p$ as above,
we have $S_1 = S$.
\end{lemma}
\begin{proof}
Since $p$ is intermediate in $(K,C,S)$, $p\not\in S$. Hence
$S_1 = S - \{ s, p \} = S - \{ s \}.$
We claim that $s\not\in S$. Suppose $s$ were a simplex in $S$.
Then, as $s$ is a proper face of $p$ (by freeness), we could take
$t=s$ in condition (iv) above and conclude that $t=s$ is a \emph{proper}
face of itself, a contradiction. Hence $s\not\in S$ as claimed. 
\end{proof}
\begin{lemma} \label{lem.c1isc}
If $(K,C,S)$ is the associated layered complex of a divided simplicial
complex $(K,S^0)$, then $C_1 = C$.
\end{lemma}
\begin{proof}
Since $p$ is intermediate in $(K,C,S)$, $p\not\in C$. Hence
$C_1 = C - \{ s, p \} = C - \{ s \}.$
(This is true even if $(K,C,S)$ does not come from a divided complex.)
We claim that if $(K,C,S)$ is associated to $(K,S^0)$, then $s\not\in C$.
Indeed, as $p\not\in C$, $p$ must then have at least one vertex
$v \in S^0$. Then $t = \{ v \}$ is a simplex of $S$ and 
$t$ is a proper face of $p$. (Note that $\dim p \geq 1$, as $p$ is
intermediate in an associated complex.)
By condition (iv) above, $\{ v \}$ is a face of $s$.
But $v\not\in C^0 = K^0 - S^0$, so $s$ has a vertex which is not in $C^0$.
It follows that $s$ cannot be in $C$, as claimed.
\end{proof}

Let us assume that $(K,C,S)$ is the associated layered complex of a divided simplicial
complex $(K,S^0)$. Then, by Lemmas \ref{lem.s1iss} and \ref{lem.c1isc},
$(K_1, C_1, S_1) = (K_1, C, S).$
It follows that
\begin{align} \label{equ.intermediacypresunderintermcoll}
\IM (K_1, C_1, S_1) 
&= \IM (K_1, C, S) 
  = \{ s\in K_1 ~|~ s\not\in C \text{ and } s\not\in S \} \nonumber \\
&= \{ s\in K ~|~ s\not\in C \text{ and } s\not\in S \} \cap K_1 
   = \IM (K,C,S) \cap K_1. 
\end{align}
Thus the notion of intermediacy is preserved under the passage from
$K$ to $K_1$: a simplex in $K_1$ is intermediate in $(K_1, C_1, S_1)$ if and only if
it is intermediate in $(K, C, S)$. 
Note that neither $s$ nor $p$ can be a vertex.
For if $s$ were a vertex, then $s = \{ v \}$ with $v\in S^0$ or $v\in K^0 - S^0 = C^0$.
In the case $v\in S^0$, we have $s\in S$ and since $s$ is a proper face of $p$,
(iv) implies that $s$ is a proper face of itself, which is impossible.
On the other hand, if $v\in C^0$, then $p=\{ u,v \}$ with $u\in S^0$
(since $p$ is intermediate, and a free face has codimension $1$ in its principal simplex). 
Then $t := \{ u \}$ is a simplex of $S$
which is a proper face of $p$ but not a proper face of $s=\{ v \}$,
contradicting (iv). Thus $s$ is not a vertex. (Since $s$ is a proper face of $p$,
the latter can of course not be a vertex either.)
Thus $K_1$ and $K$ have the same vertex set, $(K_1)^0 = K^0$.
In particular, the partition of $K^0$ into $S^0$ and its complement $C^0$ can
be regarded as a partition of $(K_1)^0$. Thus we have a well-defined
divided complex $(K_1, S^0)$.
\begin{lemma} \label{lem.k1csassoctok1s0}
The layered complex $(K_1, C, S)$ is associated to the divided complex $(K_1, S^0)$.
\end{lemma}
\begin{proof}
If $s$ is a simplex of $K_1$ whose vertices are in $S^0$, then
$s$ is in particular a simplex of $K$ whose vertices are in $S^0$, and
thus $s\in S$, as $(K,C,S)$ is associated to $(K,S^0)$.
Conversely, suppose that $s\in K_1$ is a simplex in $S$. Then $s$ is in particular
a simplex of $K$ which is in $S$, and thus all vertices of $s$ are in $S^0$,
as $(K,C,S)$ is associated to $(K,S^0)$.
Analogous reasoning applied to $C$ instead of $S$ will show that $C$ consists
precisely of those simplices in $K_1$ whose vertices are in $K^0 - S^0$.
\end{proof}

\begin{defn}
We say that the layered simplicial complex $(K_1, C, S)$ described above has been obtained
from the associated layered simplicial complex $(K,C,S)$ of a divided complex $(K,S^0)$ by
an \emph{elementary intermediate collapse}. 
\end{defn}
Lemma \ref{lem.k1csassoctok1s0} allows for an iterative execution of 
elementary intermediate collapses. Thus we may define:
\begin{defn}
The layered simplicial complex $(K_I, C, S)$ has been obtained
from the associated layered simplicial complex $(K,C,S)$ of a 
divided complex $(K,S^0)$ by
an \emph{intermediate collapse}, if $(K_I, C, S)$ is produced by a finite
sequence of elementary intermediate collapses starting from $(K,C,S)$.
\end{defn}

\begin{remark}
In the context of condition (iv), let 
$t$ be a simplex of $S$ which is a proper face of $p\in \IM (K,C,S)$.
Then, by (iv), $t$ is a proper face of $s$, which is itself a proper face of $p$.
Therefore,
\[ \dim t \leq \dim p - 2. \]
This can be regarded as an echo of the pseudomanifold condition:
In a triangulation of an $n$-dimensional pseudomanifold, a simplex is
principal if and only if it has dimension $n$. 
(For if a principal simplex had smaller dimension, then
it would have to be the face of some $n$-simplex and so could not be principal.)
Then if $t$ is any simplex in
the singular set $\Sigma$, 
\[ \dim t \leq \dim \Sigma \leq n-2 = \dim p -2. \]
(We note however, that $p$ in the present situation does not actually 
have a free face $s$.)
\end{remark}

\begin{defn}
Let $(K,C,S), (K',C',S')$ be layered complexes associated to divided
simplicial complexes. We say that 
$(K',C',S')$ has been obtained from $(K,C,S)$ by an
\emph{elementary layered collapse}, if there is an 
elementary $C$-, $S$-, or intermediary collapse of a simplex
in $(K,C,S)$ that yields $(K',C',S')$.
In that case, we say that $(K,C,S)$ is obtained from
$(K',C',S')$ by an \emph{elementary layered expansion}.
Let $(X,\Sigma) = (|K|, |S|)$ and $(X',\Sigma') = (|K'|, |S'|)$
be the polyhedral pairs determined by the layered complexes.
In the above situation, we then say that
$(X',\Sigma')$ has been obtained from $(X,\Sigma)$ by an
\emph{elementary stratified collapse} and $(X,\Sigma)$ from
$(X',\Sigma')$ by an \emph{elementary stratified expansion}.
A \emph{layered collapse} from $(K,C,S)$ to $(K',C',S')$ is
a finite sequence of elementary layered collapses;
similarly for \emph{layered expansions}. For the polyhedral pairs,
this leads to the notion of \emph{stratified collapses} and
\emph{stratified expansions}.
A \emph{layered formal deformation} from $(K,C,S)$ to 
$(K',C',S')$ is a finite sequence of transformations, each of which is
either a layered collapse or a layered expansion, starting from
$(K,C,S)$ and ending with $(K',C',S')$.
The associated sequence of polyhedral pairs is called a
\emph{stratified formal deformation} from $(X,\Sigma)$ to
$(X',\Sigma')$.
\end{defn}

Classical elementary collapses can be executed in any order,
since a free face cannot ever become non-free by performing
classical collapses.
Our stratified theory involves three different types of collapses
and we must carefully investigate the effects of sequentially ordering these
types, particularly as this is relevant for the algorithmic
implementation (Section \ref{sec.implementation}).
\begin{lemma} \label{lem.codim2faces}
If $p$ is a simplex and $q<p$ a face of codimension at least $2$, then 
for any proper face $s<p$ with $q<s$, there
exists a second proper face $t<p,$ $t\not= s,$ such that $q<t$.
\end{lemma}
\begin{proof}
The face $s$ is obtained from $p$ by omitting at least one vertex $u$.
The face $q$ is obtained from $s$ by omitting at least one vertex $v$, $v\not= u$.
Then $t = q \cup \{ u \}$ does the job.
\end{proof}
\begin{prop} \label{prop.nosintermthennos}
Let $(K,C,S)$ be the associated layered complex of a divided simplicial complex.
If $(K,C,S)$ admits no $S$-collapse and $(K',C,S)$ is obtained from $(K,C,S)$
by an intermediate collapse, then $(K',C,S)$ still does not admit an $S$-collapse.
\end{prop}
\begin{proof}
The complex $K'$ is of the form
\[ K' = K - \{ f,p \}, \]
where $f$ and $p$ satisfy the conditions (i)--(iv) of an intermediate collapse.
In particular, $p\in \IM (K,C,S)$, $p$ is principal in $K$, and $f<p$ is free in $K$.
We argue by contradiction: Suppose that $(K',C,S)$ \emph{does} admit an $S$-collapse.
Then there is a simplex $q\in S$ which is principal in $K'$ and has a
proper face $s<q$ which is free in $K'$.
Since there is no $S$-collapse possible in $(K,C,S),$ either $q$ is not
principal in $K$, or $q$ has no free face in $K$.

Suppose first that $q$ is not principal in $K$.
Then $q<p$ or $q<f$. 
If $q<p$, then by condition (iv) for intermediate collapses,
$q<f$ (as $q\in S$). So we may assume $q<f$.
Thus $q$ is a face of codimension at least $2$ of $p$.
By Lemma \ref{lem.codim2faces}, 
there exists a proper face $r<p$ such that $r\not= f$ and $q<r$.
Note that $r$ is a simplex of $K'$.
Thus $q$ is not principal in $K'$, a contradiction.
We conclude that $q$ must already be principal in $K$ and
$q$ has no free face in $K$.
In particular, $s$ is \emph{not} free in $K$, but \emph{is} free in $K'$.
Therefore, $s<p$ or $s<f$.
If $s<p$, then by condition (iv) for intermediate collapses,
$s<f$ (as $s\in S$). So we may assume $s<f$.
Thus $s$ is a face of codimension at least $2$ of $p$.
By Lemma \ref{lem.codim2faces}, 
there exists a proper face $r<p$ such that $r\not= f$ and $s<r$.
Note that $r$ is a simplex of $K'$.
If $r=q$, then $q<p$, so $q$ is not principal in $K$, a contradiction.
Hence $r\not= q$.
But then $s<r\in K'$ and $s<q\in K'$ and $r\not= q$, which contradicts
the freeness of $s$ in $K'$.
Thus $q$ and $s$ as above cannot exist and $(K',C,S)$ does not
admit an $S$-collapse, as was to be shown.
\end{proof}

The following example illustrates the process of layered collapses
and shows in particular that intermediate collapses, carried out
after all possible $S$- and $C$-collapses have been performed, may enable
new $C$-collapses (though by the proposition, they cannot enable new $S$-collapses).
\begin{example} \label{expl.intermccoll}
Let $p = \{ s_0, s_1, c_0, c_1 \}$ be a $3$-simplex with vertices
$s_0, s_1, c_0, c_1$ and let $K$ be the simplicial complex generated by $p$.
Let $S^0 = \{ s_0, s_1 \}$, determining a divided complex $(K,S^0)$.
The complex $S$ of the associated layered complex $(K,C,S)$ is the simplicial
subcomplex generated by the $1$-simplex $\{ s_0, s_1 \}$ and $C$ is the
subcomplex generated by the $1$-simplex $\{ c_0, c_1 \}$, see
Figure \ref{fig.expleintermcoll}.
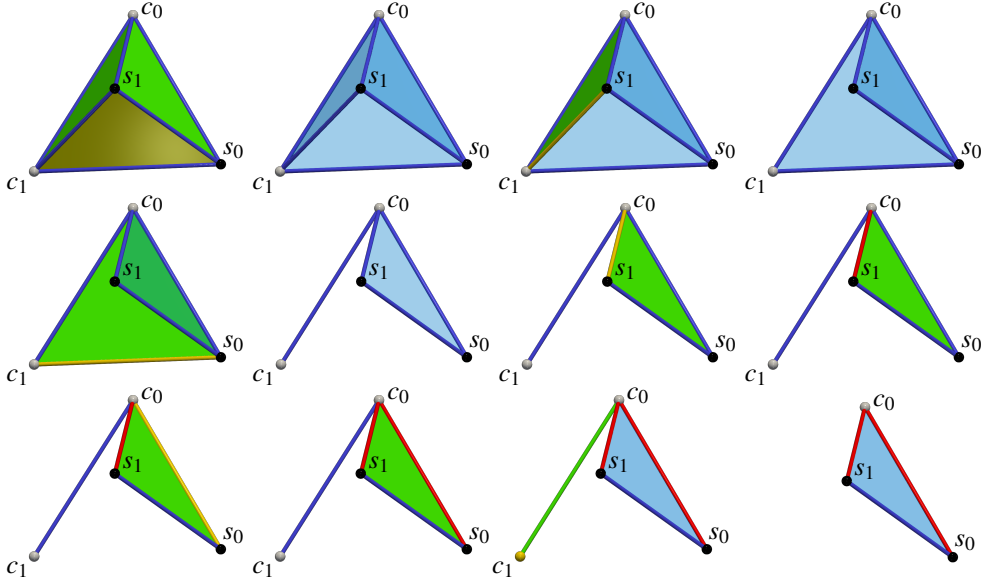
\begin{figure}[!h]
    \centering
    \begin{minipage}{0.245\textwidth}
        \centering
        \includegraphics[width=1\textwidth]{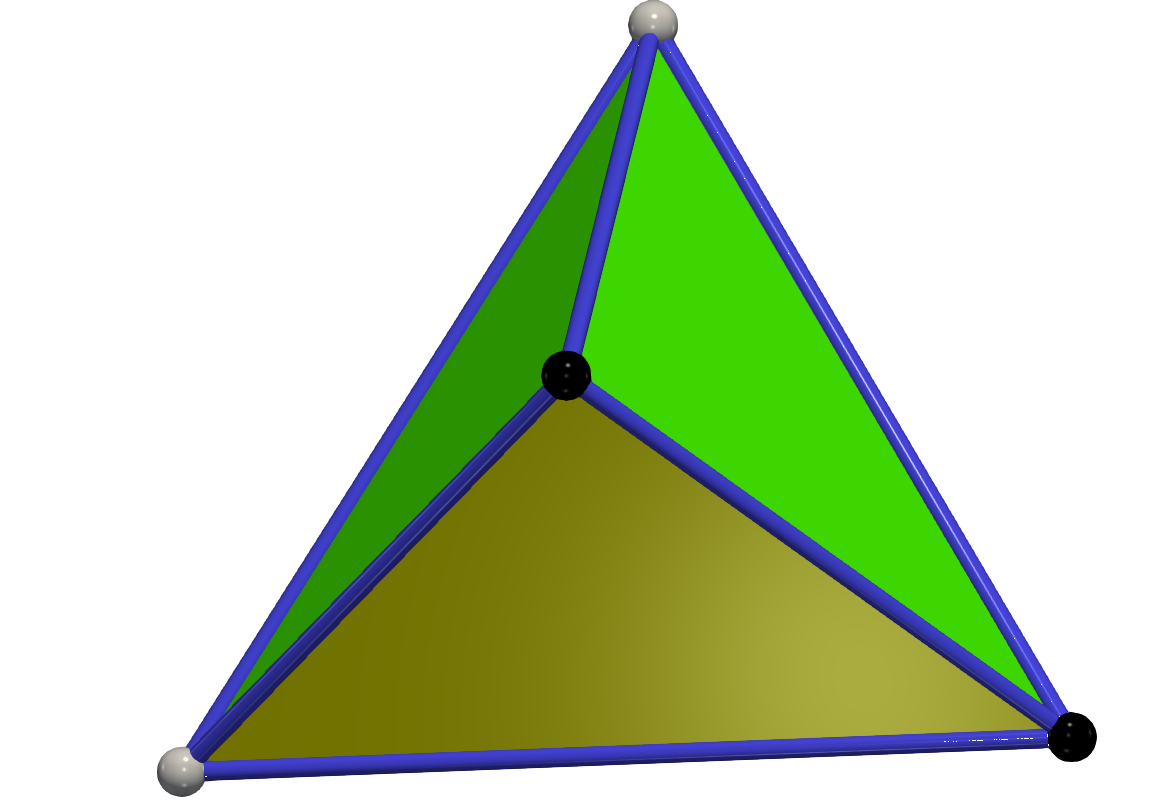}
        \put(-5,9){$s_0$}
        \put(-43,36){$s_1$}
        \put(-87,-4){$c_1$}
        \put(-36,62){$c_0$}
    \end{minipage}\hfill
    \begin{minipage}{0.245\textwidth}
        \centering
        \includegraphics[width=1\textwidth]{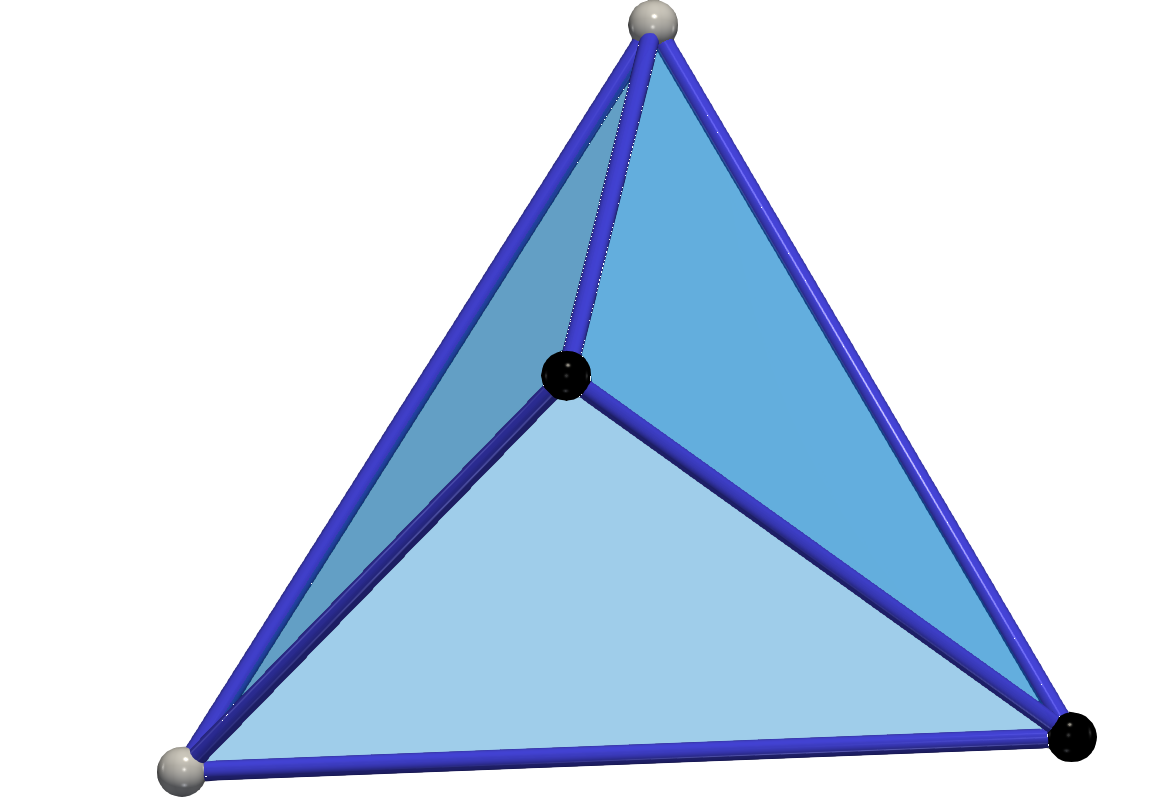}
        \put(-5,9){$s_0$}
        \put(-43,36){$s_1$}
        \put(-87,-4){$c_1$}
        \put(-36,62){$c_0$}
    \end{minipage}
    \begin{minipage}{0.245\textwidth}
        \centering
        \includegraphics[width=1\textwidth]{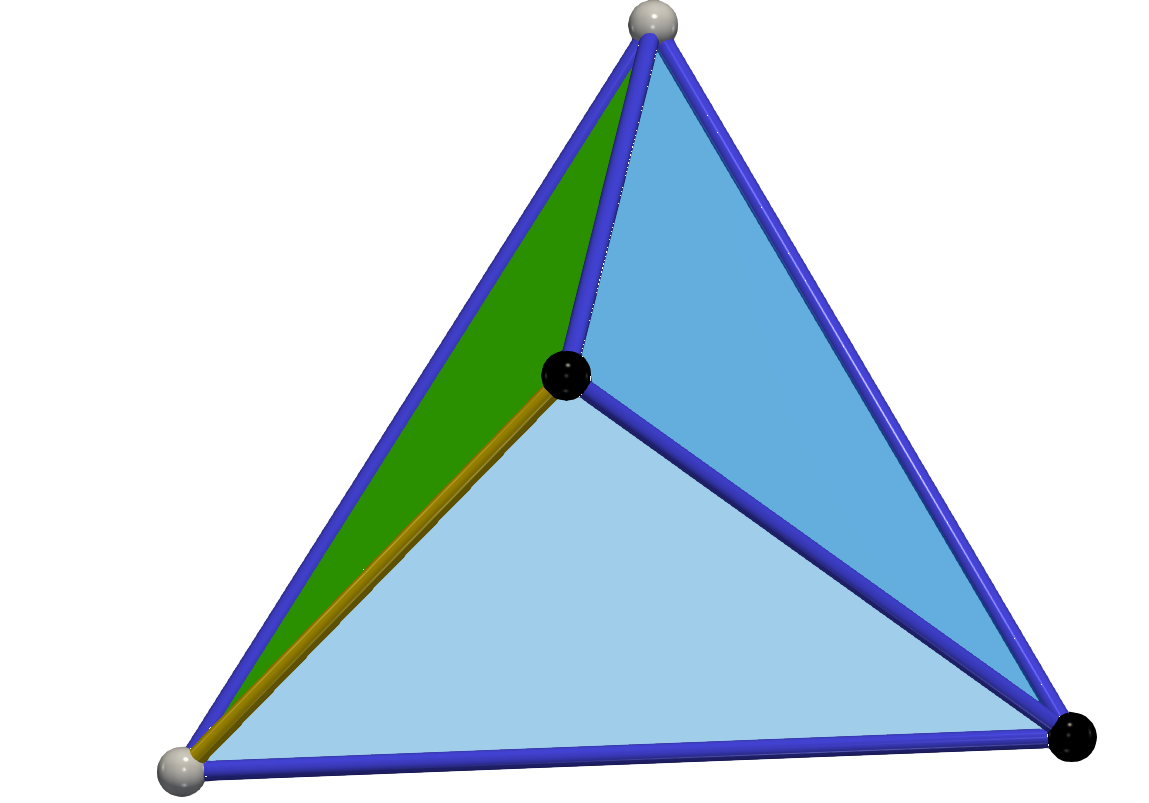}
        \put(-5,9){$s_0$}
        \put(-43,36){$s_1$}
        \put(-87,-4){$c_1$}
        \put(-36,62){$c_0$}
    \end{minipage}
  	      \centering
    \begin{minipage}{0.245\textwidth}
        \centering
        \includegraphics[width=1\textwidth]{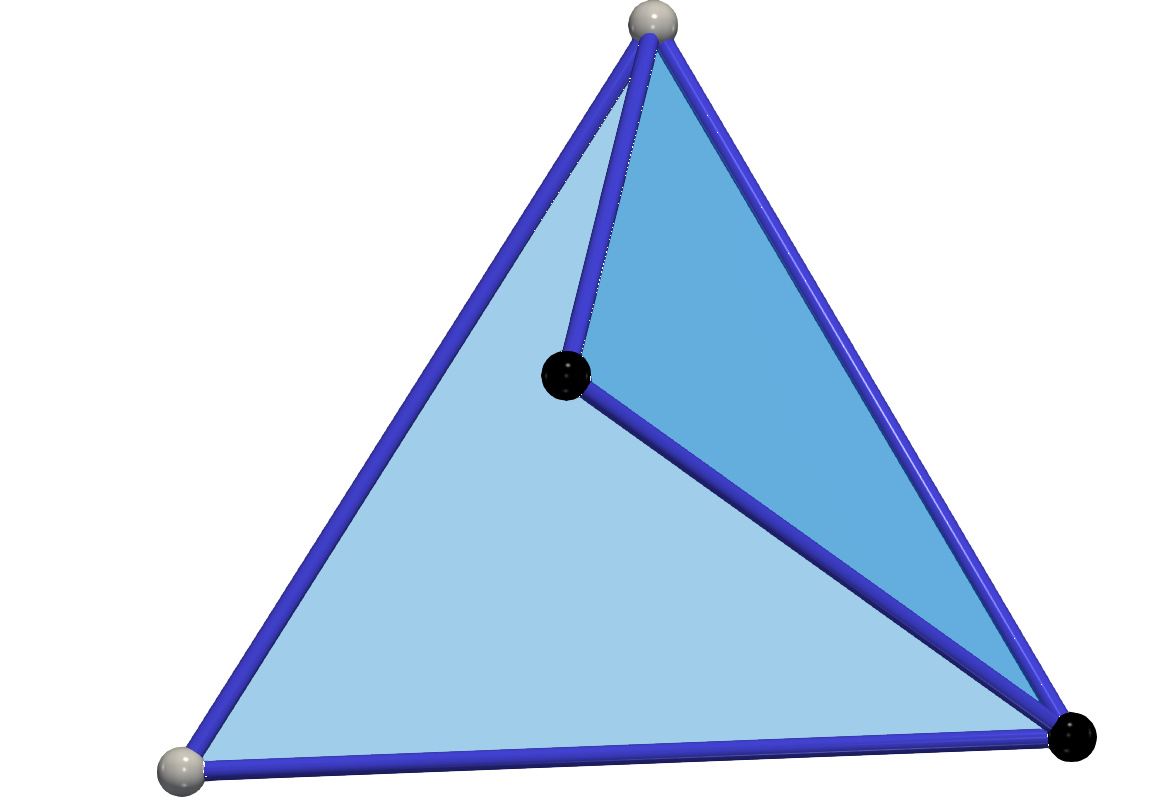}
        \put(-5,9){$s_0$}
        \put(-43,36){$s_1$}
        \put(-87,-4){$c_1$}
        \put(-36,62){$c_0$}
    \end{minipage}\hfill
    \begin{minipage}{0.245\textwidth}
        \centering
        \includegraphics[width=1\textwidth]{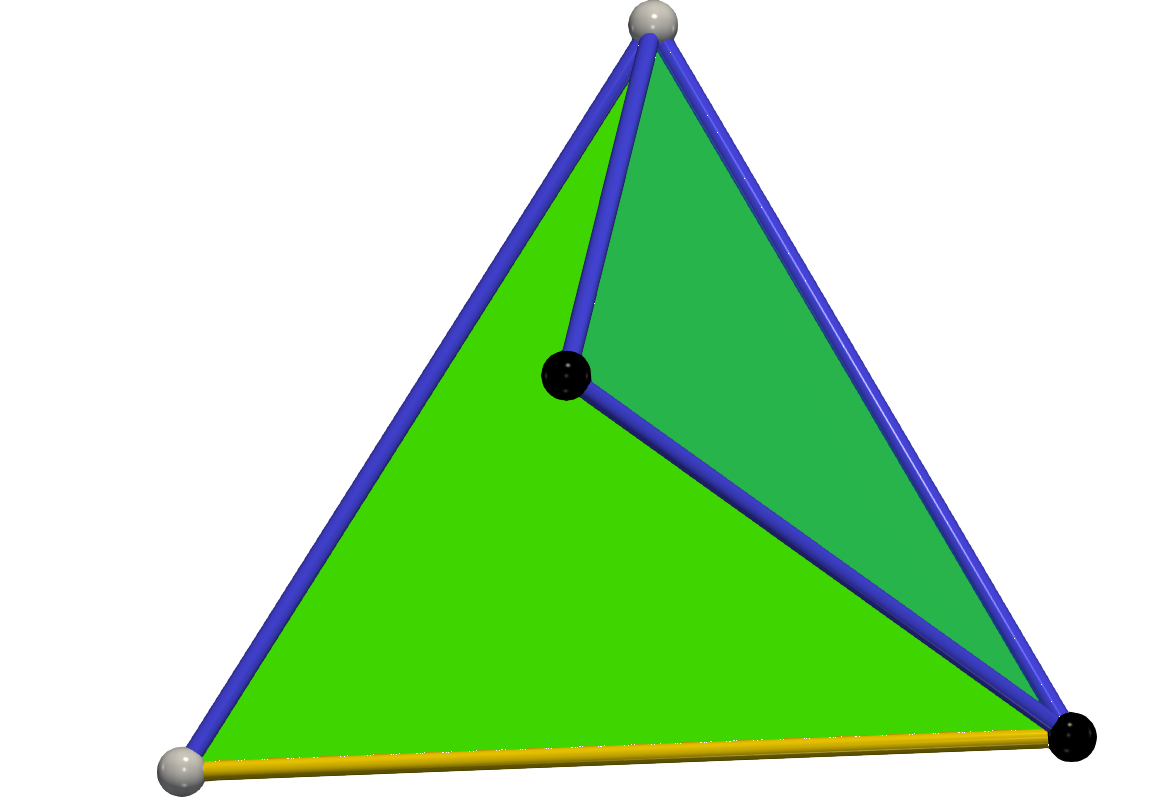}
        \put(-5,9){$s_0$}
        \put(-43,36){$s_1$}
        \put(-87,-4){$c_1$}
        \put(-36,62){$c_0$}
    \end{minipage}
    \begin{minipage}{0.245\textwidth}
        \centering
        \includegraphics[width=1\textwidth]{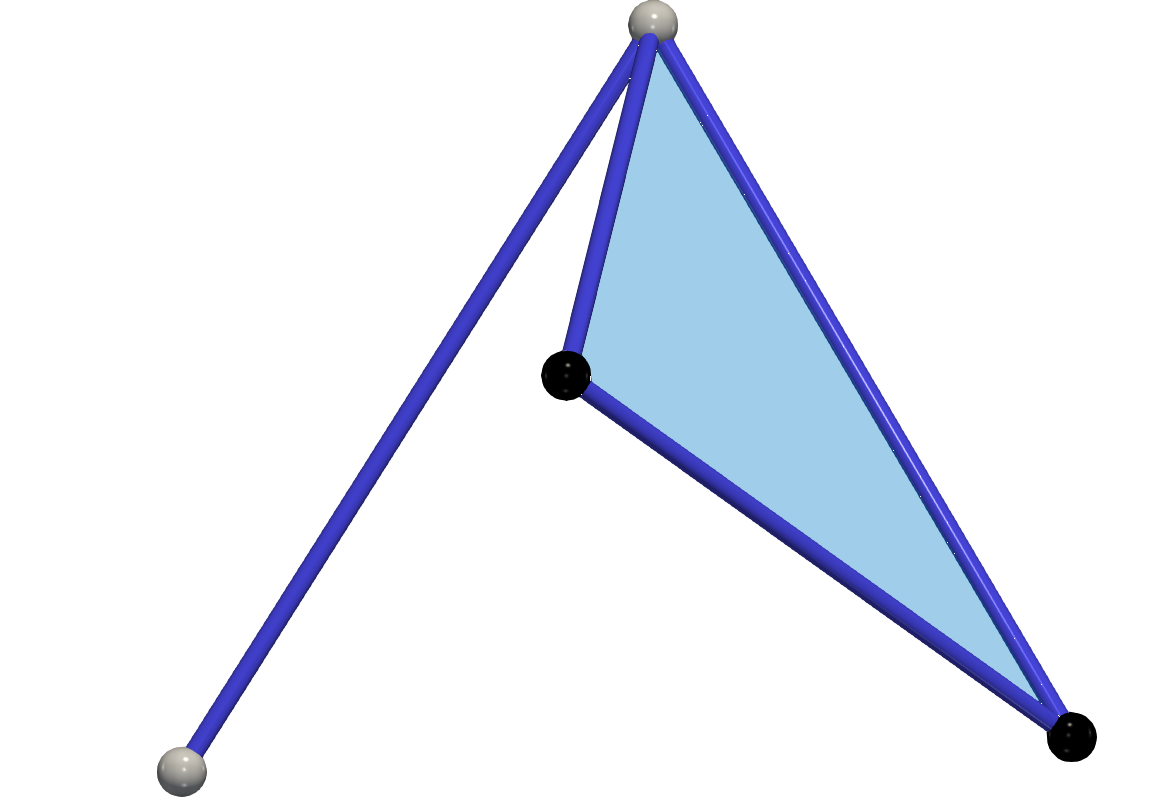} 
        \put(-5,9){$s_0$}
        \put(-43,36){$s_1$}
        \put(-87,-4){$c_1$}
        \put(-36,62){$c_0$}
    \end{minipage}  
        \centering
    \begin{minipage}{0.245\textwidth}
        \centering
        \includegraphics[width=1\textwidth]{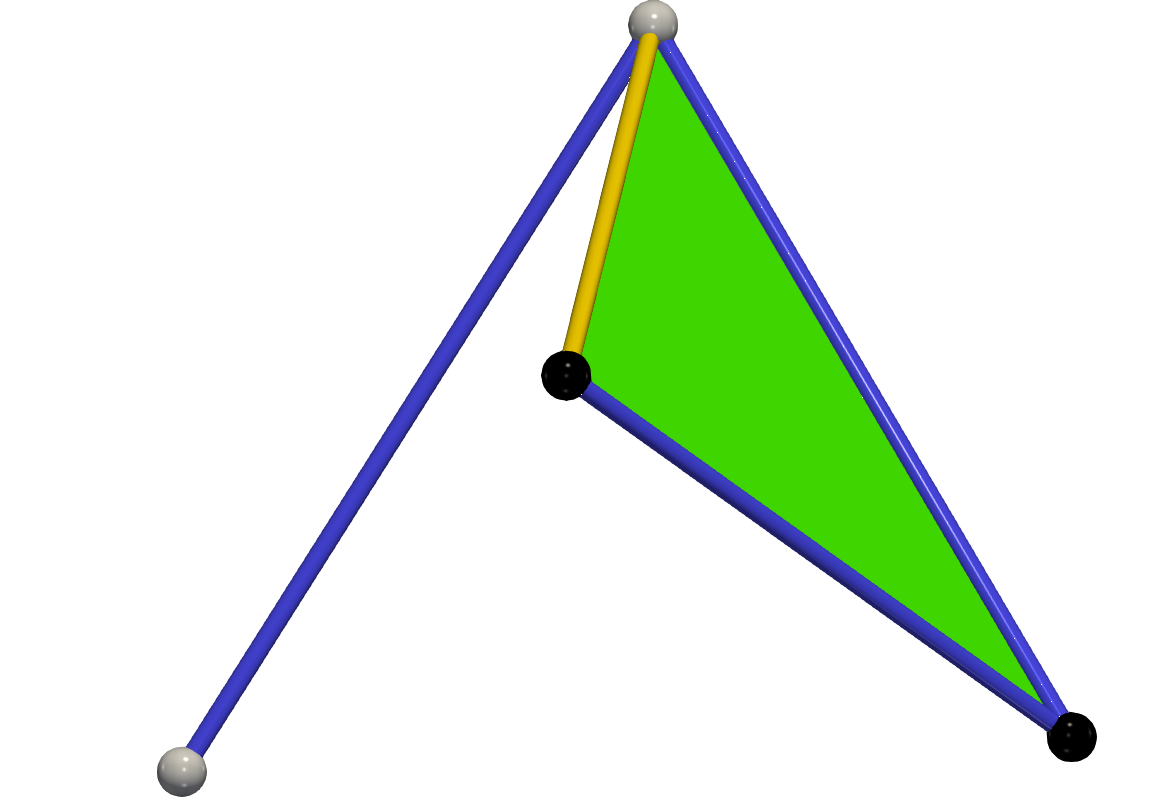}
        \put(-5,9){$s_0$}
        \put(-43,36){$s_1$}
        \put(-87,-4){$c_1$}
        \put(-36,62){$c_0$}
    \end{minipage}\hfill
    \begin{minipage}{0.245\textwidth}
        \centering
        \includegraphics[width=1\textwidth]{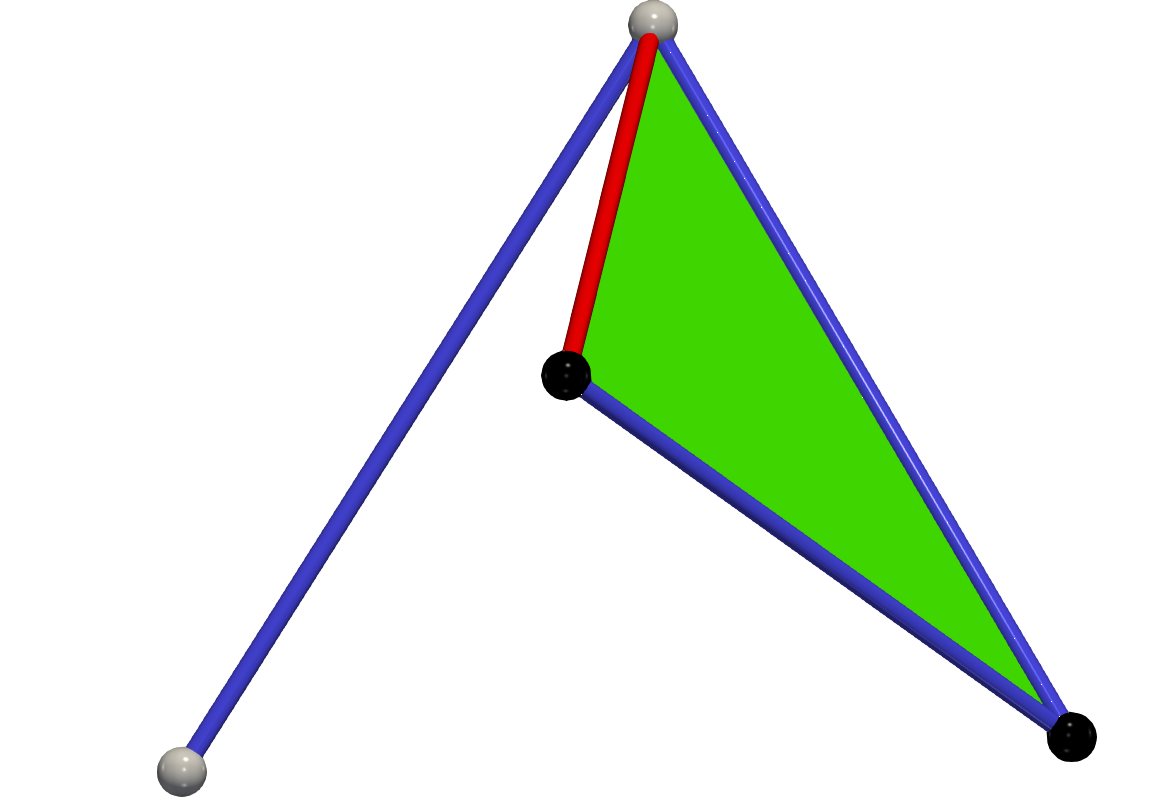}
        \put(-5,9){$s_0$}
        \put(-43,36){$s_1$}
        \put(-87,-4){$c_1$}
        \put(-36,62){$c_0$}
    \end{minipage}
    \begin{minipage}{0.245\textwidth}
        \centering
        \includegraphics[width=1\textwidth]{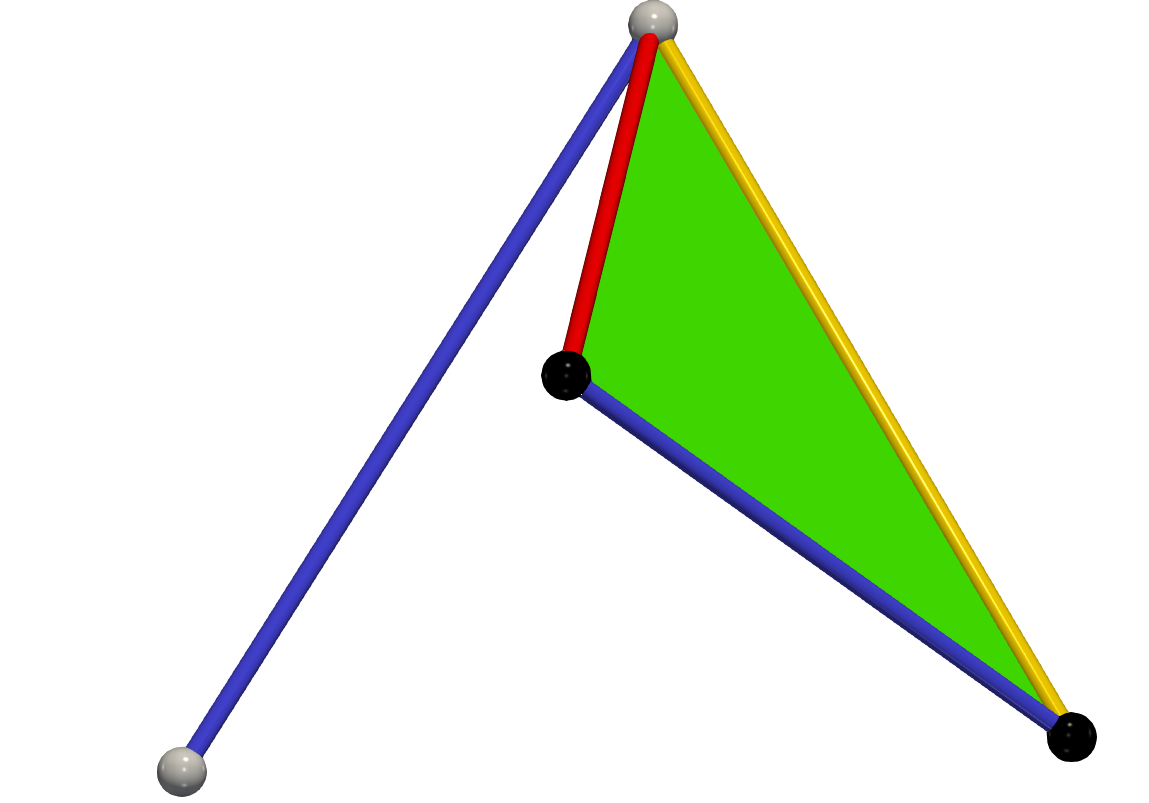}
        \put(-5,9){$s_0$}
        \put(-43,36){$s_1$}
        \put(-87,-4){$c_1$}
        \put(-36,62){$c_0$}
    \end{minipage}  
        \centering
    \begin{minipage}{0.245\textwidth}
        \centering
        \includegraphics[width=1\textwidth]{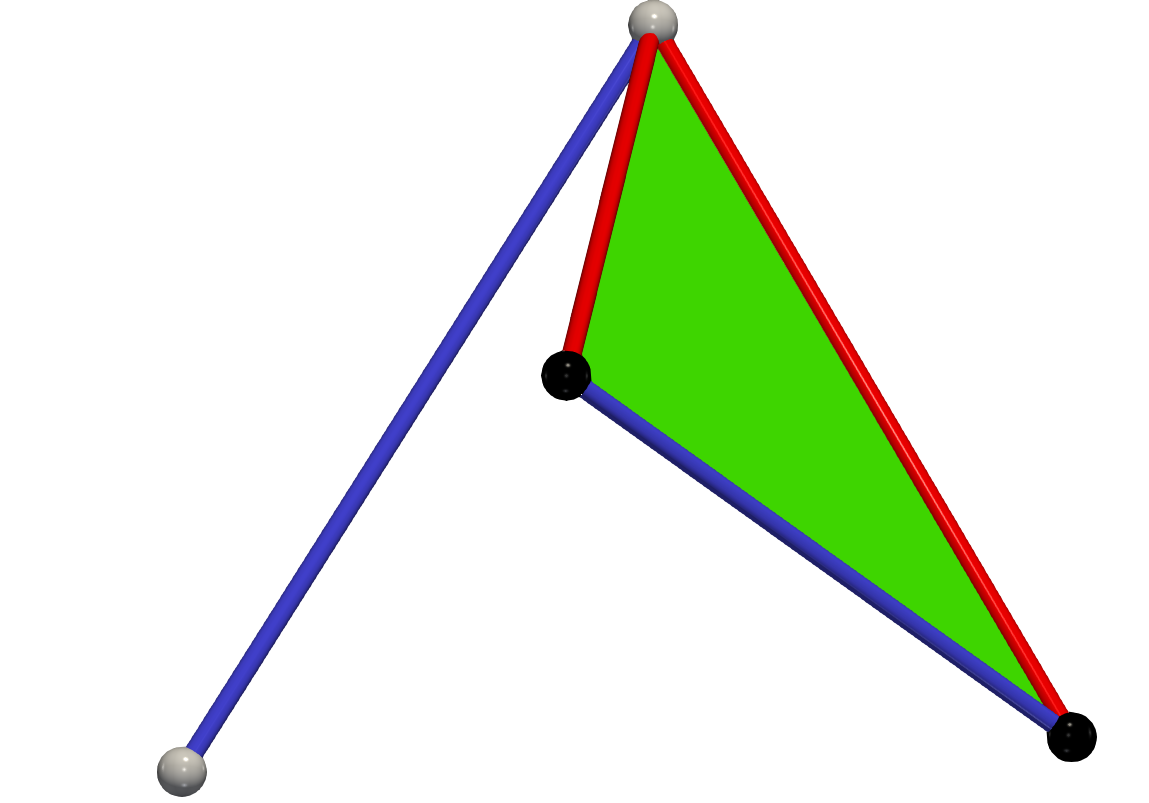}
        \put(-5,9){$s_0$}
        \put(-43,36){$s_1$}
        \put(-87,-4){$c_1$}
        \put(-36,62){$c_0$}
    \end{minipage}\hfill
    \begin{minipage}{0.245\textwidth}
        \centering
        \includegraphics[width=1\textwidth]{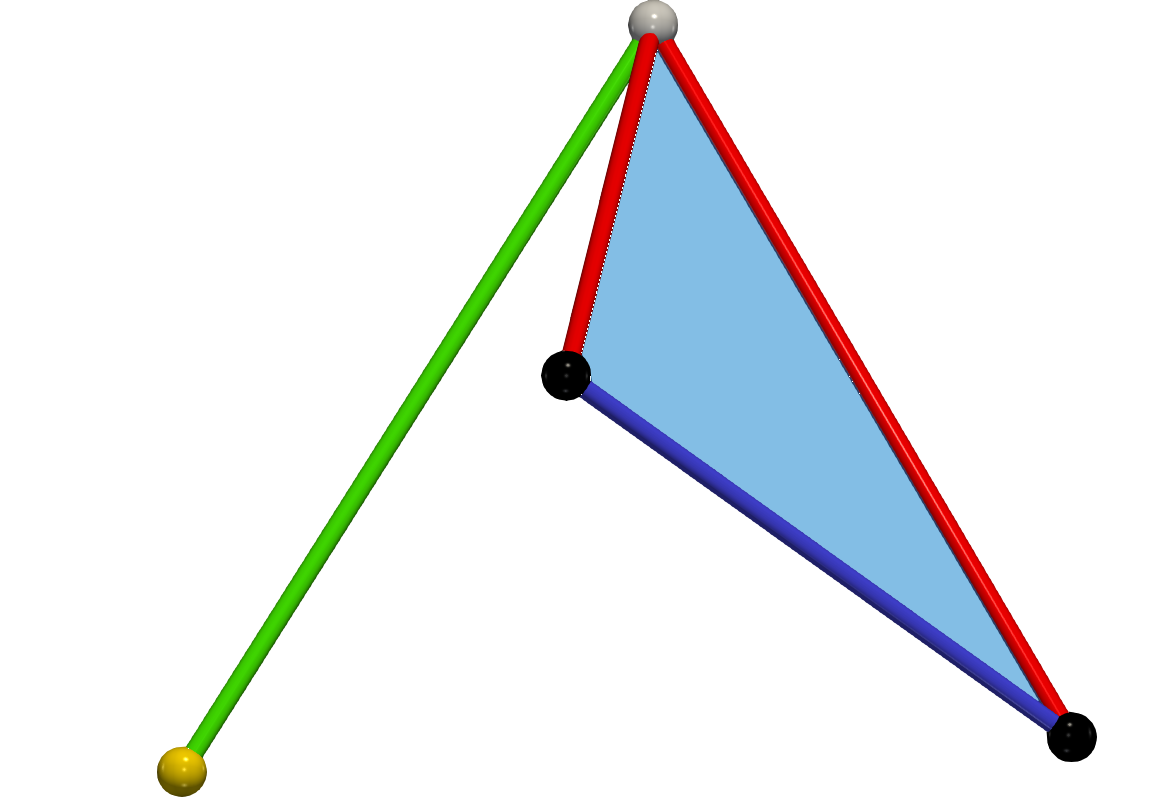}
        \put(-5,9){$s_0$}
        \put(-43,36){$s_1$}
        \put(-87,-4){$c_1$}
        \put(-36,62){$c_0$}
    \end{minipage}
    \begin{minipage}{0.245\textwidth}
        \centering
        \includegraphics[width=1\textwidth]{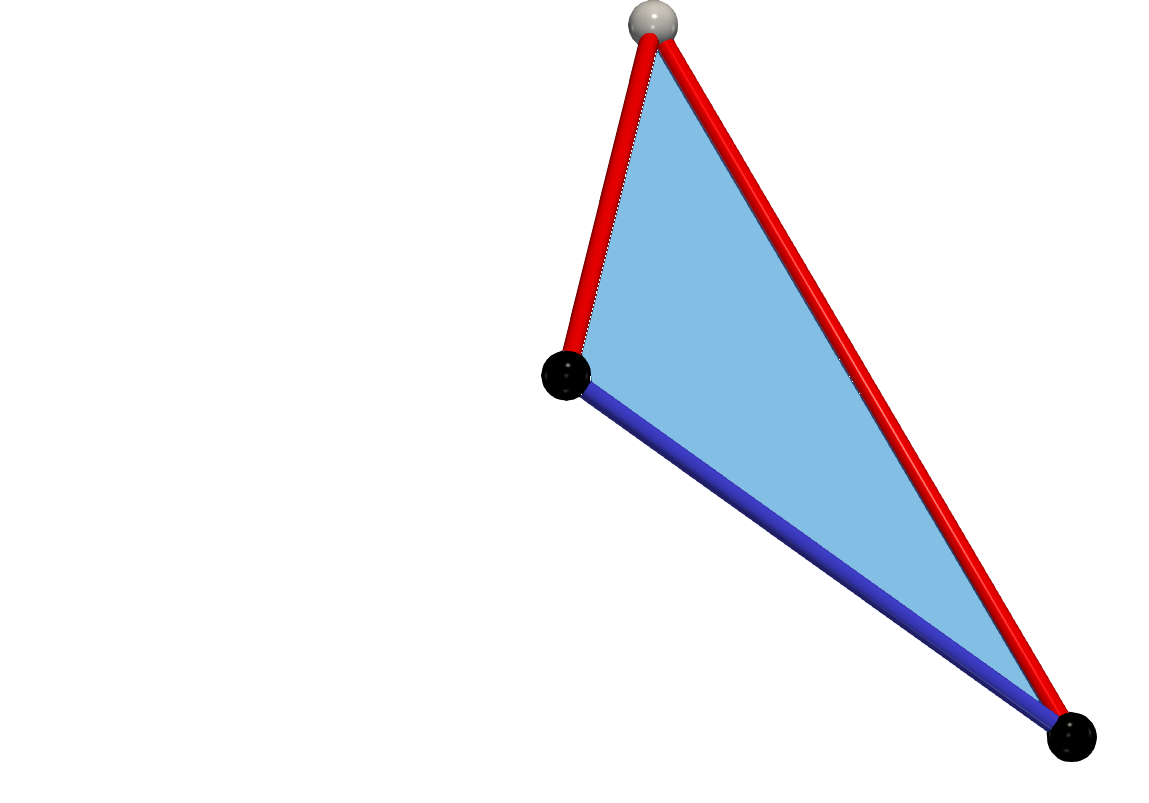}
        \put(-5,9){$s_0$}
        \put(-43,36){$s_1$}
        \put(-36,62){$c_0$}
    \end{minipage}  
    \caption{Intermediate and $C$-Collapses of $K$ to a layered spine. 
    Free faces (yellow), principal cofaces (green), 
    intermediate free faces not fulfilling condition (iv) (red).}
    \label{fig.expleintermcoll}
\end{figure}

No $S$-collapse is possible in $(K,C,S)$, since the $1$-simplex
$\{ s_0, s_1 \}$, while principal in $S$, is not principal in $K$.
Similarly, no $C$-collapse is possible.
However, there is one possible intermediate collapse:
The simplex $p = \{ s_0, s_1, c_0, c_1 \}$ is intermediate in $(K,C,S)$,
principal in $K$, and
$s=\{ s_0, s_1, c_1 \}$ is a free face of $p$.
We have to verify condition (iv):
The simplices $t$ in $S$ are $\{ s_0, s_1 \},$ $\{ s_0 \}$ and $\{ s_1 \}$.
All of these are proper faces of $p$, so the conclusion of (iv) must
be checked. Indeed, all these $t$ are proper faces of $s = \{ s_0, s_1, c_1 \}$.
Thus there is an admissible intermediate collapse
$K\searrow K_1 = K - \{ s, p \}$.
The complex $K_1$ is $2$-dimensional, with three $2$-simplices.
No $C$-collapse or $S$-collapse is feasible in $(K_1, C, S)$.
 
The simplex $p = \{ s_1, c_0, c_1 \}$ is intermediate in $(K_1, C,S)$,
principal in $K_1$, and
$s=\{ s_1, c_1 \}$ is a free face of $p$.
We verify condition (iv):
The simplices $t= \{ s_0, s_1 \}$ and $t= \{ s_0 \}$ in $S$
are not faces of $p$. But $t=\{ s_1 \}$ is a proper face of $p$,
so the conclusion of (iv) must be checked for this $t$. Indeed,
$t=\{ s_1 \}$ is a proper face of $s = \{ s_1, c_1 \}$.
Thus there is an admissible intermediate collapse
$K_1 \searrow K_2 = K_1 - \{ s, p \}$.
The complex $K_2$ has precisely two $2$-simplices.

The simplex $p = \{ s_0, c_0, c_1 \}$ is intermediate in $(K_2, C,S)$,
principal in $K_2$, and
$s=\{ s_0, c_1 \}$ is a free face of $p$.
We verify condition (iv):
The simplices $t= \{ s_0, s_1 \}$ and $t= \{ s_1 \}$ in $S$
are not faces of $p$. But $t=\{ s_0 \}$ is a proper face of $p$,
so the conclusion of (iv) must be checked for this $t$. Indeed,
$t=\{ s_0 \}$ is a proper face of $s = \{ s_0, c_1 \}$.
Thus there is an admissible intermediate collapse
$K_2 \searrow K_3 = K_2 - \{ s, p \}$.
The complex $K_3$ has precisely one $2$-simplex.

Note that $(K_3, C,S)$ does not admit an intermediate collapse,
but it does admit a $C$-collapse: The $1$-simplex $\{ c_0, c_1 \}$
is in $C$, is principal in $K_3$ and has the free face
$\{ c_1 \}$. The corresponding elementary $C$-collapse
yields a layered complex $(K_{3C}, C', S)$ with
$C' = \{ \{ c_0 \} \}$.
Although the complex $K_{3C}$ (forgetting the layer structure)
is collapsible to a point, the layered complex $(K_{3C}, C', S)$
admits no $C$-collapse, no $S$-collapse, and no intermediate collapse.
\end{example}

\begin{prop} \label{prop.nointermcthennointerm}
Let $(K,C,S)$ be the associated layered complex of a divided simplicial
complex $(K,S^0)$. Suppose that $(K,C,S)$ does not admit an intermediate
collapse. If $(K',C',S)$ is obtained from $(K,C,S)$ by a
$C$-collapse, then $(K', C', S)$ still does not admit an intermediate
collapse.
\end{prop}
\begin{proof}
The complex $K'$ has the form
$K' = K - \{ f,p \},$ $p\in C$, $p$ principal in $K$, and
$f<p$ free in $K$. Since $f$ and $p$ are in $C$, we have
$C' = C - \{ f,p \}$.
We argue by contradiction:
Suppose an intermediate collapse in $(K',C',S)$ were possible.
Then there would exist simplices $g,q$ in $K'$ such that\\

\noindent (a) $q\in \IM (K',C',S),$\\
(b) $q$ is principal in $K'$,\\
(c) $g<q$ with $g$ free in $K'$,\\
(d) for every $t\in S$: if $t<q$, then $t<g$.\\

\noindent Statement (a) implies that $q\not\in S$ and $q\not\in C' = C-\{ f,p \}$.
As $q$ lies in $K'$ it must be different from $f$ and $p$.
Thus $q\not\in C$. Hence is already intermediate in $(K,C,S$,
$q\in \IM (K,C,S)$.

Suppose $q$ were not principal in $K$.
Then, as it \emph{is} principal in $K'$ by (b), $q<p$ or $q<f$.
If $q<f$, then in particular $q<p$, so we may assume $q<p$.
But this places $q$ into $C$, a contradiction to $q\in \IM (K,C,S)$.
Therefore, $q$ is already principal in $K$.

Is $g$ free in $K$? If not, then (since it \emph{is} free in $K'$)
$g<p$ or $g<f$. If $g<f$, then $g<p$, so we may assume the latter,
which implies that $g\in C$. So $g$ has the form
$g=\{ c_0,\ldots, c_k \}$ for vertices $c_0,\ldots, c_k \in C^0 = K^0 - S^0$.
Since $g$ is a codimension one face of $q$ (Lemma \ref{lem.freehascodim1}),
$q$ must have the form
$q =\{ c_0,\ldots, c_k, s \}$. Since $q$ is intermediate, the vertex $s$
must be in $S^0$.
Take $t:= \{ s \}$. Then $t\in S$ and $t<q$.
By (d), $t<g$ and we reach a contradiction.
We conclude that $g$ is free in $K$.

We have seen that $q\in \IM (K,C,S),$ $q$ is principal in $K$, and
$g$ is a free face of $q$ in $K$.
But $(K,C,S)$ does not admit an intermediate collapse.
Consequently, condition (iv) for intermediate collapses must be
violated for $g,q$ in $(K,C,S)$. Thus there must exist a
$t\in S$ such that $t<q$ but $t$ is not a proper face of $g$ in $K$.
Such a $t$ must be in $K'$, and we arrive at a contradiction to (d).
Therefore, an intermediate collapse in $(K',C',S)$ is not possible.
\end{proof}

\begin{defn}
Let $(K,C,S)$ be the associated layered simplicial complex of
a divided complex.
A layered simplicial complex $(K', C', S')$ is called a
\emph{layered spine} of $(K,C,S)$ if
$(K', C', S')$ can be obtained from $(K,C,S)$ by a layered collapse
and $(K',C',S')$ does not admit any further layered collapse
(though it may admit further ordinary collapses).
In this case, we say that the polyhedral pair
$(X',\Sigma') = (|K'|, |S'|)$ is a \emph{stratified spine}
of $(X,\Sigma) = (|K|,|S|)$.
\end{defn}

As Example \ref{expl.intermccoll} shows, a layered spine of a layered
complex $(K,C,S)$ need not be a spine (in the ordinary sense) of the underlying 
complex $K$.
However, if the polyhedron of the underlying simplicial complex of a layered spine is
a pseudomanifold, then it \emph{will} be a spine, by Lemma \ref{lem.pseudomfdisspine}.
Note that the polyhedron of $K_C$ in Example \ref{expl.intermccoll} is
not a pseudomanifold.
Section \ref{sec.examples} contains several example calculations of
stratified spines, and their intersection homology,
in simplicial complexes coming from point data.

\section{Freely Orthogonal Deformation Retractions}
\label{sec.freelyorthogdefretr}

The polyhedron of an elementary simplicial collapse is
a deformation retract of the polyhedron of the original complex.
In order to obtain stratified homotopy equivalence, the particular
choice of deformation retraction plays a role. We shall here construct
suitable explicit retractions, called \emph{freely orthogonal}, 
and investigate their properties. This material is needed only
in the proofs of our main results in Section \ref{sec.formaldefhtpytypeih}.

Let $K$ be a simplicial complex, $p$ a principal simplex of $K$,
and $f$ a free face of $p$. With $K'$ denoting the result of the
associated elementary collapse, we shall construct a particular 
deformation retraction
\[ H: |K| \times I \longrightarrow |K| \]
onto $|K'|$. Essentially, the idea is to project orthogonally from
the free face.
Let $v, f_1, \ldots, f_m$ be the vertices of $p$, where the $f_i$
are the vertices of the free face $f$.
Let $e_i$ denote the $i$-th standard basis vector of $\real^m$.
Identify $v$ with the origin $0\in \real^m$ and $f_i$ with $e_i$
for all $i=1,\ldots, m$. Via barycentric coordinates, this 
identifies $|p|$ with the convex hull of 
$0,e_1,\ldots, e_m$ in $\real^m$.
Under this identification, $|p|$ is given by
\[ \{ (x_1,\ldots, x_m)\in \real^m
  ~|~ x_i \geq 0,~ \sum_{i=1}^m x_i \leq 1 \}. \]
The polyhedron $|f|$ of the free face corresponds to the
convex hull of $e_1, \ldots, e_m$,
\[ \{ (x_1,\ldots, x_m)\in \real^m
  ~|~ x_i \geq 0,~ \sum_{i=1}^m x_i = 1 \}. \]
The affine hyperplane $H\subset \real^m$ containing the free face is
of the form $e_1 + L$, where $L \subset \real^m$ is the linear
subspace spanned by the basis
$e_2 - e_1,~ e_3 - e_1,~ \ldots,~ e_m - e_1.$
Let $\ell$ be the line (through the origin) in $\real^m$
spanned by the vector
\[ n = \sum_{i=1}^m e_i = (1,1,\ldots, 1), \]
which is a normal vector to $L$ (with respect to the
standard Euclidean inner product on $\real^m$).
Let $\partial |p|$ denote the boundary of $|p|$ (i.e. the union
of all proper faces) and set
\[ \Lambda = \partial |p| - \operatorname{interior}|f|. \]
Thus $\Lambda$ is the union of all intersections of 
$|p|$ with coordinate hyperplanes $x_j =0$.
\begin{lemma} \label{lem.retractionxtoy}
For every point $x\in |p|$, the line through $x$ orthogonal to $H$
intersects $\Lambda$ in a unique point $y$, that is,
$(x+\ell) \cap \Lambda = \{ y \}.$
\end{lemma}
\begin{proof}
Given $x=(x_1,\ldots, x_m)\in |p|$
points on $x+\ell$ have the form
\[ x+tn = (x_1 +t,\ldots, x_m +t),~ t\in \real. \]
A point of $|p|$ lies in $\Lambda$ precisely when at least one of its
coordinates vanishes. So if $y$ is on $x+\ell$ and on $\Lambda$, then
one of its coordinates, say $x_j +t$ vanishes, so that $t= -x_j$.
Thus
\begin{equation} \label{equ.coordy}
y = (x_1 - x_j, x_2 -x_j, \ldots, x_m - x_j). 
\end{equation}
But all of these coordinates must be nonnegative; hence
\begin{equation} \label{equ.determinexj}
x_j = \min \{ x_1,\ldots, x_m \}. 
\end{equation}
This shows that $y$, if it exists, is unique.
Now for existence, define $y$ by
(\ref{equ.determinexj}) and (\ref{equ.coordy}).
Then $y$ is plainly on $x+\ell$ and on $\Lambda$.
\end{proof}
Lemma \ref{lem.retractionxtoy} shows that there is a well-defined map
$r: |p| \longrightarrow \Lambda,~ x\mapsto y.$
Explicitly, this map is given by
\[ r(x_1,\ldots,x_m) = (x_1 - x_j, x_2 -x_j, \ldots, x_m - x_j),~ 
   x_j = \min \{ x_1,\ldots, x_m \}. \]
This shows that $r$ is continuous.
If $x\in \Lambda,$ then $x_j =0$ and thus $r(x)=x$.
Thus $r$ is a retraction onto $\Lambda$.

\begin{lemma} \label{lem.rxinbfimplxinbf}
If $x\in |p|$ is a point such that $r(x)\in \partial |f|,$ then
$x\in \partial |f|$.
\end{lemma}
\begin{proof}
In terms of coordinates, $\partial |f|$ is characterized by
\[ \partial |f| = \Lambda \cap |f| = \{ (x_1,\ldots, x_m) ~|~
  x_i \geq 0,~ \sum x_i =1, \text{ and }
  \exists j: x_j =0 \}. \]
Suppose $x\in |p|$ retracts onto the boundary of $|f|$, i.e.
$r(x)\in \partial |f|$.
Then the coordinates
$(x_1 - x_j, \ldots, x_m - x_j)$, $x_j = \min \{ x_i \},$ of $r(x)$ satisfy
\[ \sum_{i=1}^m (x_i - x_j) =1, \]
that is,
\[ \sum_i x_i = 1+mx_j. \]
But $x\in |p|$ ensures that $\sum x_i \leq 1$.
Consequently, $x_j =0$.
(Note that $m\geq 1$, since $m= \dim p$ and $p$ has a proper face $f$.)
This implies that $x\in \Lambda$.
Furthermore, it follows that $\sum x_i =1$, i.e. $x\in |f|$.
We conclude that $x\in \Lambda \cap |f| = \partial |f|$.
\end{proof}

Define 
$H: |p| \times I \longrightarrow |p|$
to be the straight line homotopy
\[ H(x,t) = (1-t)x + t r(x). \]   
This is well-defined as $|p|$ is convex.
Then $H$ is a homotopy from $H_0 = \id$ to $H_1 = r$
such that if $x\in \Lambda,$ then
$H(x,t) = (1-t)x + tx = x,$
that is, $H$ is a deformation retraction onto $\Lambda$.
This can be readily extended to a deformation retraction
\[ H: |K| \times I \to |K| \]
of $|K|$ onto $|K'|$ by setting $H(x,t)=x$ for all $t\in I$ and all
$x\in |K| - |p|$.
We shall refer to this $H$ as the \emph{freely orthogonal deformation
retraction} associated to $p$ and $f$.
Note that all the proper faces of $f$ are contained in $\Lambda$.
Hence $H_t$ is the identity on the proper faces of $f$.

Let $(K,C,S)$ be a layered simplicial complex.
\begin{lemma} \label{lem.scollstrat}
Let $(K,C,S) \searrow (K_S,C,S')$ be an elementary $S$-collapse
of a principal simplex $p\in S$ using the free face $f\in S$.
Then the associated freely orthogonal deformation retraction
$H: |K|\times I\to |K|$ maps
\[ H_t (|K| - |S|) = |K|- |S| \text{ and }
   H_t (|S|) \subset |S| \]
for all $t\in I$.
\end{lemma}
\begin{proof}
If $x\in |S|-|p|$, then
$H_t (x)=x\in |S|$ for all $t$. If, on the other hand, $x\in |p| \subset |S|$
then $H_t (x) \in |p| \subset |S|$. Hence $H_t (|S|) \subset |S|$ for all $t$.
Points $x\in |K|-|S|$ satisfy $x\in |K|-|p|$, so that
$H_t (x)=x$. This shows that $H_t (|K| - |S|) = |K|- |S|$.
\end{proof}

\begin{lemma} \label{lem.ccollstrat}
Let $(K,C,S) \searrow (K_C,C',S)$ be an elementary $C$-collapse
of a principal simplex $p\in C$ using the free face $f\in C$.
Then the associated freely orthogonal deformation retraction
$H: |K|\times I\to |K|$ maps
\[ H_t (|K| - |S|) \subset |K|- |S| \text{ and }
   H_t (|S|) = |S| \]
for all $t\in I$.
\end{lemma}
\begin{proof}
As $|p|\subset |C|$ and $|S|\cap |C|=\varnothing$, the complex $|S|$ is contained
in $|K|-|p|$, on which $H_t$ is the identity for all $t$.
We conclude that $H_t (|S|) = |S|$.
If $x\in |K|-(|S| \cup |p|)$, then again $H_t (x)=x\in |K|-|S|$.
If, on the other hand, $x\in |p| \subset |K|-|S|$, then
$H_t (x) \in |p| \subset |K|-|S|$.
\end{proof}

\begin{lemma} \label{lem.intermcollstrat}
Let $(K,C,S)$ be the associated layered complex of a divided complex
$(K,S^0)$.
Let $(K,C,S) \searrow (K_I,C,S)$ be an elementary intermediate collapse
of a principal simplex $p\in \IM (K,C,S)$ using the free face $f$.
Then the associated freely orthogonal deformation retraction
$H: |K|\times I\to |K|$ maps
\[ H_t (|K| - |S|) \subset |K|- |S| \text{ and }
   H_t (|S|) \subset |S| \]
for all $t\in I$.
\end{lemma}
\begin{proof}
We shall first show that
\begin{equation} \label{equ.scappinbndryf}
|S| \cap |p| \subset \partial |f|.
\end{equation}
Indeed, given a point $x\in |S| \cap |p|,$ there is a unique
face $s\in S$ of $p$ that contains $x$ in its interior.
This face must be a proper face of $p$, since if $s=p,$ then
$p\in S$, contradicting $p\in \IM (K,C,S)$.
By condition (iv) for intermediate collapses,
$s$ is a proper face of $f$. Hence 
$x\in |s| \subset \partial |f|$, establishing (\ref{equ.scappinbndryf}).

As $\partial |f| \subset \Lambda$ and $r$ is a retract onto $\Lambda$,
it follows in particular that $r(x)=x$ for points $x\in |S| \cap |p|$
and hence that $H_t (x) = (1-t)x+tr(x)=x$ for such $x$ and all $t$.
So $H_t (|S|\cap |p|) \subset |S|$ for all $t$.
If $x\in |S|-|p|$, then $H_t (x)=x\in |S|$ by definition.
This proves $H_t (|S|)\subset |S|$ for all $t$.

It remains to verify $H_t (|K|-|S|) \subset |K|-|S|.$
Let $x$ be a point in $|K|-|S|$.
If $x\not\in |p|$, then $H_t (x)=x\in |K|-|S|$.
So suppose that $x\in |p|$. If $r(x)=x$, then again $H_t (x)=x$ and 
we are done. So assume that $r(x)\not= x$.
Suppose it were then true that $r(x)\in |S|$.
Then $r(x) \in |p|\cap |S|$ and hence $r(x)\in \partial |f|$ by (\ref{equ.scappinbndryf}).
Lemma \ref{lem.rxinbfimplxinbf} implies that $x\in \partial |f|$.
As $\partial |f| \subset \Lambda$ and $r$ is a retract onto $\Lambda$,
we deduce that $r(x)=x$, contradicting $r(x)\not= x$.
Therefore $r(x)\not\in |S|$. Thus $H_t (x)\not\in |S|$ for $t=0,1$.

We claim that if $t>0$, then $H_t (x)\not\in |f|$.
By contradiction: Suppose $(1-t)x + tr(x) \in |f|$, so that
\[ \sum_i ((1-t)x_i + t(x_i - x_j))=1,~ x_j = \min \{ x_1,\ldots, x_m \}. \]
Note that $x_j >0$, since we know that $x\not\in \Lambda$
(as $r(x)\not= x$).
It follows that
\[ \sum_i x_i = 1 + mtx_j. \]
Since $t>0,$ $m\geq 1$ and $x_j >0$, this implies $\sum x_i >1$, 
contradicting $\sum x_i \leq 1$ ($x\in |p|$).
Thus $H_t (x)\not\in |f|$ as claimed.
Now if $H_t (x)$ were in $|S|$ for $t>0$, then it would have to be
in $\partial |f| \subset |f|$ by (\ref{equ.scappinbndryf}), which 
we have just proved to be impossible.
Hence $H_t (x)\not\in |S|$, as was to be shown.
\end{proof}

\section{Stratified Formal Deformations, Homotopy Type and Intersection Homology}
\label{sec.formaldefhtpytypeih}

Our main result relates the layered formal deformation type of a complex 
to the stratified homotopy type of associated filtered spaces.
As a corollary, we conclude that the intersection homology is
preserved by layered formal deformations.
In particular, the intersection homology of any two stratified spines
(which may well not be homeomorphic) will be isomorphic. 
Recall that for ease of exposition we will illustrate our results 
only for filtrations of type (\ref{equ.twostrataspaces}).
\begin{lemma} \label{lem.ellaycollstrathe}
Let $(K,C,S)$ and $(K',C',S')$ be the associated layered simplicial complexes
of divided simplicial complexes, and suppose that
the polyhedra $X=|K|$ and $X'=|K'|$ are filtered spaces
with respective singular sets $\Sigma = |S|,$ $\Sigma' = |S'|$
whose formal codimensions coincide, $\codim_X \Sigma = \codim_{X'} \Sigma'$.
If $(K',C',S')$ is obtained from $(K,C,S)$ by an
elementary layered collapse,
then $(X,\Sigma)$ and $(X',\Sigma')$ are stratified homotopy equivalent.
\end{lemma}
\begin{proof}
Let $(K,C,S) \searrow (K',C',S')$ be an elementary layered collapse
of a principal simplex $p$ using the free face $f$ of $p$.
Let $H: |K| \times I \to |K|$ be the freely orthogonal deformation
retraction associated to $p$ and $f$, satisfying
$H_1 (|K|) \subset |K'|$; on $|p|$, $H_1$ is given by
the retraction $r$ of Section \ref{sec.freelyorthogdefretr}.
We claim that the inclusion
\[ (|K'|,|S'|) \stackrel{i}{\hookrightarrow} (|K|,|S|) \]
is a stratified homotopy equivalence with stratified  homotopy inverse
\[ (|K|,|S|) \stackrel{r}{\longrightarrow} (|K'|,|S'|). \]
To prove the claim, we establish first that both
$i$ and $r=H_1$ are codimension-preserving stratified maps.
First, a preliminary remark.
In the case of an $S$-collapse,
\begin{equation} \label{equ.sk1kk1capsk} 
S' = K' \cap S 
\end{equation}
by (\ref{equ.sikicaps});
in the case of a $C$-collapse,
$S' = S.$
The latter equality holds also in the case of an intermediate
collapse, by Lemma \ref{lem.s1iss}.
When $S' = S$, then
$K' \cap S = K' \cap S' = S'.$
Hence (\ref{equ.sk1kk1capsk}) holds in all three cases.
Now the inclusion $i: |K'| \hookrightarrow |K|$ maps
$|S'|$ to $|K| \cap |S'|$.
Using (\ref{equ.sk1kk1capsk}),
\[ |K| \cap |S'| = |K| \cap |K'| \cap |S|
\subset |S|, \]
so $i$ maps the singular sets correctly.
Since
$H_1 (|K|) \subset |K'|,$ and $H_t (|S|) \subset |S|$
for all $t$, we have
\[ r (|S|) \subset |K'| \cap |S| = |S'|, \]
so $r$ maps the singular sets correctly as well.

If $x\in |K'| - |S'|,$ then (\ref{equ.sk1kk1capsk})
shows that $x\not\in |S|$.
Thus $i (|K'| - |S'|) \subset |K| - |S|$ as claimed.
Let $x$ be a point in $|K| - |S|,$ and 
suppose that $r(x)\in |S'|$.
Then (\ref{equ.sk1kk1capsk}) shows that $r(x)\in |S|$.
But by Lemmas 
\ref{lem.scollstrat}, \ref{lem.ccollstrat}, and \ref{lem.intermcollstrat},
$r(x) = H_1 (x)\in |K|-|S|$,
a contradiction since $|S| \subset |K|$.
Thus $r(x)\not\in |S'|$ so that
$r(|K| - |S|) \subset |K'| - |S'|$ as required.
We have shown that $i$ and $r$ are stratified maps.
They are codimension-preserving by assumption.

Since $r$ is a retraction, $r\circ i = \id_{|K'|}$.
It remains to be shown that 
$i\circ r$ is stratified homotopic to the identity on $|K|$.
A stratified homotopy is given by the above
freely orthogonal deformation
retraction $H$:
By Lemmas 
\ref{lem.scollstrat}, \ref{lem.ccollstrat}, and \ref{lem.intermcollstrat},
$H$ maps $|S| \times I$ into $|S|$ and $(|K|-|S|)\times I$ into
$|K|-|S|$ in all three cases ($S$-, $C$-, and intermediate collapse).
\end{proof}

\begin{thm} \label{thm.main}
Let $(K,C,S)$ and $(K',C',S')$ be the associated layered simplicial complexes
of divided simplicial complexes, and suppose that
the polyhedra $X=|K|$ and $X'=|K'|$ are filtered spaces
with respective singular sets $\Sigma = |S|,$ $\Sigma' = |S'|$
whose formal codimensions coincide, $\codim_X \Sigma = \codim_{X'} \Sigma'$.
If there exists a layered formal deformation between
$(K,C,S)$ and $(K',C',S')$,
then $(X,\Sigma)$ and $(X',\Sigma')$ are stratified homotopy equivalent.
\end{thm}
\begin{proof}
Let
\[ (K,C,S) = (K_0,C_0,S_0)
    \to (K_1,C_1,S_1)
    \to \cdots
    \to (K_m,C_m,S_m) = (K',C',S') \]
be a layered formal deformation between $(K,C,S)$ and $(K',C',S')$,    
where each arrow indicates an elementary layered collapse or 
layered expansion.
The polyhedral pair
$(X_i, \Sigma_i) := (|K_i|, |S_i|),$
$i=0,\ldots, m,$ becomes a filtered space by declaring the
\emph{formal} codimensions of the singular sets to be
$\codim_{X_i} \Sigma_i := \codim_X \Sigma$.
Then Lemma \ref{lem.ellaycollstrathe}
is applicable to $(K_i,C_i,S_i) \to (K_{i+1},C_{i+1},S_{i+1})$
and yields a stratified homotopy equivalence
$(X_i, \Sigma_i) \simeq (X_{i+1}, \Sigma_{i+1})$ for
every $i=0,\ldots, m-1$.
The composition of these stratified homotopy equivalences is
a stratified homotopy equivalence
\[ (X,\Sigma) = (X_0, \Sigma_0) \simeq (X_m, \Sigma_m) 
    = (X',\Sigma') \]
by Lemma \ref{lem.compstrathes}.
\end{proof}

As a special case of the theorem, we deduce that any two
stratified spines of a layered complex lie in the same 
classical stratified homotopy type:
\begin{cor} \label{cor.stratifiedspines}
Let $(K,C,S)$ be the associated layered simplicial complex of
a divided simplicial complex.
If $(K', C', S')$ and $(K'', C'', S'')$ are layered spines of $(K,C,S)$
whose polyhedra $X=|K'|$ and $Y=|K''|$ are filtered spaces 
with respective singular sets $\Sigma_X = |S'|,$ $\Sigma_Y = |S''|$
whose formal codimensions coincide,
then the stratified spines 
$(X,\Sigma_X)$ and $(Y,\Sigma_Y)$ are stratified homotopy equivalent.
\end{cor}

Theorem \ref{thm.main} has the following consequence for 
the combinatorial deformation stability of
singular intersection homology:
\begin{cor} \label{cor.ihstable}
Let $(K,C,S)$ and $(K',C',S')$ be the associated layered simplicial complexes
of divided simplicial complexes, and suppose that
the polyhedra $X=|K|$ and $X'=|K'|$ are filtered spaces
with respective singular sets $\Sigma = |S|,$ $\Sigma' = |S'|$
whose formal codimensions coincide.
If there exists a layered formal deformation between
$(K,C,S)$ and $(K',C',S')$,
then
$IH^{\bar{p}}_* (X) \cong IH^{\bar{p}}_* (Y)$ for every $\bar{p}$.
\end{cor}
\begin{proof}
This follows from Theorem \ref{thm.main} in view
of Proposition \ref{prop.ihstratheinv}.
\end{proof}
To compute the singular intersection homology groups algorithmically
via simplicial intersection chains, one may use
Propositions \ref{prop.simplihcomputesplih} 
and \ref{prop.plihcomputessingih}.
From the point of view of a given filtered Vietoris-Rips type
complex $K$ associated to data points near a filtered space $X$,
our results lead to conditions under which the complex $K$ itself
can be used to compute the intersection homology of $X$:
\begin{cor} \label{cor.ihfromvr}
Let $(K,S^0)$ be a divided simplicial complex and $(X,\Sigma)$
a filtered space. (In practice, $K$ could be a
Vietoris-Rips, \v{C}ech, etc. type complex associated to points near
$X$, with points of $S^0$ near $\Sigma$.)
If there exists a layered formal deformation from the associated
layered complex $(K,C,S)$ to a layered  complex $(K',C',S')$
whose filtered space $(|K'|,|S'|)$, with $\codim |S'|=\codim_X \Sigma$, 
is stratified homotopy equivalent
to $(X,\Sigma)$, then, taking formal codimension 
$\codim_{|K|} |S| := \codim |S'|,$
the intersection homology of the filtered space
$(|K|,|S|)$ computes the intersection homology
of $(X,\Sigma)$.
\end{cor}
As pointed out earlier, these codimensions are not \emph{a priori}
known, but may in fact be revealed by spines.
The corollary shows in particular that if a Vietoris-Rips type complex $K$
has a stratified spine which is, say, stratified homeomorphic to $X$,
then $IH^{\bar{p}}_* (X)$ may already be computed from 
$IH^{\bar{p}}_* (K)$ by using
the appropriate codimension.

\section{Implementation}
\label{sec.implementation}

Given a set of data points, we used standard Delaunay-Vietoris-Rips
methods to generate simplicial complexes, but 
our stratified collapse methods are applicable to any
filtered polyhedron and thus one could in addition investigate
the behavior of \v{C}ech complexes, witness complexes, etc. 
The Delaunay-Vietoris-Rips complex is the restriction of the Vietoris-Rips 
complex to simplices of the Delaunay triangulation. 
Using Delaunay triangulations to generate simplicial complexes from 
sets of data points is a common approach in topological data analysis, 
see e.g. \cite{alpha}, \cite{harish}, \cite{baueredelsbrunner}.
To illustrate stratified formal deformations, we focus here
on stratified spines.
A possible pseudocode sequencing of the three types of 
elementary layered collapse
operations is shown in Algorithm \ref{alg.layspine}.
\begin{algorithm}
\begin{flushleft} 
 \textbf{Input:} \text{$(K,C,S)$ layered 
   complex of a divided simplicial complex}\\ 
   \textbf{Output:} \text{$(K,C,S)$ layered complex 
       with no further layered collapses possible} 
\end{flushleft}        
  \caption{LayeredSpine$(K, C, S)$}
  \begin{algorithmic}[1]	
\While{$S$-collapse possible in $(K,S)$}
	\State $(K,S) \gets$ Collapse($K$,$S$)
\EndWhile

\While{$C$-collapse possible in $(K,C)$}
	\State $(K,C) \gets$ Collapse($K$,$C$)
\EndWhile 
\State $IM \gets$ IntermediateSimplices$(K,C,S)$
\While{intermediate collapse possible in $(K,IM,S)$}
	\State $(K,IM) \gets$ IMCollapse($K$,$IM$,$S$)
\EndWhile

\While{$C$-collapse possible in $(K,C)$}
	\State $(K,C) \gets$ Collapse($K$,$C$)
\EndWhile 

\State \textbf{return} $(K,C,S)$
  \end{algorithmic}
  \label{alg.layspine}
\end{algorithm}
The code blocks 1--3 and 4--6 can be interchanged, as
$S$-collapses and $C$-collapses are commuting operations.
This shows in particular that after step 6, no new $S$-collapses
can become possible.
After step 10, no $S$-collapses can become possible
by Proposition \ref{prop.nosintermthennos}. But $C$-collapses
may become possible after intermediate collapses, as
Example \ref{expl.intermccoll} shows. Hence the code block 11--13 is required.
After step 13, no $S$-collapse can be possible, for if it were, then
it could be commuted to be executed between line 10 and 11, but we
already know that after step 10, no $S$-collapse is possible.
Furthermore, after step 13, no intermediate collapse is possible
by Proposition \ref{prop.nointermcthennointerm}.
Therefore, after step 13, the resulting layered simplicial complex
is in fact a layered spine.
In step 7, the set $IM$ of intermediate simplices of $(K,C,S)$ is computed,
which is straightforward --- include all simplices of $K$ that have at
least one vertex in $S$ and at least one vertex in $C$.
According to Lemma \ref{lem.impresundercscoll},
$IM$ can be computed earlier, but it is of course advantageous to compute it
as late as possible, since the previous collapses reduce the
search space.
The function Collapse($K$,$L$) executes 
elementary collapses in $K$ only of simplices contained in a subcollection 
$L \subset K$. This is suited for $S$- and $C$-collapses. 
Intermediate collapses require a different treatment, detailed in 
Algorithm \ref{alg.coll}.
We included only the pseudocode for intermediate collapses. For $S$- and 
$C$-collapses the code looks almost identical, excluding step 5, 
which tests
condition (iv) of an intermediate collapse by calling the
function isAdmissible($S$,$s$,$p[0]$). Algorithm \ref{alg.coll}
makes implicit use of the fact, established earlier
(in Lemmas \ref{lem.s1iss} and  \ref{lem.c1isc}), 
that a free face of a principal
intermediate simplex such that (iv) holds must itself be intermediate.
Conversely, any coface (in particular a principal one) 
of an intermediate simplex is obviously intermediate itself.
The function Princ$(IM,s,K)$ searches for cofaces in $IM$ of a simplex
$s$ that are principal in $K$. 
(By the previous remark, one need not search in all of $K$,
only in $IM$.)
If it finds none or more than one, it returns the empty
list, otherwise a list $p$ of length $1$ 
containing the unique principal coface $p[0]$ (in which case $s$
is free).
If desired, it is after execution of 
\ref{alg.layspine}
algorithmically possible to check whether
the polyhedron of the layered spine is a pseudomanifold:
Let $n$ be the highest dimension of a simplex in the complex.
Then check whether every $(n-1)$-simplex is the face of precisely
two $n$-simplices and whether every simplex is the face of some $n$-simplex. \\

\begin{algorithm}
  \hspace*{\algorithmicindent} 
\begin{flushleft} 
 \textbf{Input:} \text{$K$ simplicial complex, 
    $S$ subcomplex of $K$, $IM$ intermediate simplices} \\
   \textbf{Output:} \text{$K$ simplicial complex, $IM$ intermediate simplices} 
\end{flushleft}   
  \caption{IMCollapse$(K,IM,S)$}
  \begin{algorithmic}[1]

    \State $I \gets IM$  
    	\For{$s$ \textbf{in} $I$}    	
    		\State $p \gets$ Princ$(IM,s,K)$ 
    		\If{ nonempty$(p)$}
    			\If{isAdmissible$(S,s,p[0])$}
					\State Remove $p[0]$ from $K$
					\State Remove $p[0]$ from $IM$
					\State Remove $s$ from $K$
					\State Remove $s$ from $IM$\\
					\Comment At this point, the new $IM$ is indeed
					           the set \\
					\Comment of intermediate simplices of the new $K$
					           by (\ref{equ.intermediacypresunderintermcoll}). 
				\EndIf    	
    		\EndIf

		\EndFor    
    
\State \Return $(K,IM)$
  \end{algorithmic}
  \label{alg.coll}
\end{algorithm}

Furthermore, there exist algorithms to compute (simplicial) intersection homology,
even persistent intersection homology, \cite{bendichharer}, \cite{blrs}.
We followed these to code the calculation of simplicial intersection
Betti numbers. Persistent intersection homology has been implemented by
Bastian Rieck in his \texttt{Aleph} package.
In principle, calculating intersection homology groups uses the same
matrix operations that may be used to compute ordinary simplicial homology. 
The difference lies in the process that sets up the simplicial intersection 
chain complex. The definition of $IC^{\bar{p}}_{*}(-) \subset C_{*}(-)$ 
requires us to check which simplicial chains are $\bar{p}$-allowable for a 
given perversity $\bar{p}$. Note that it is not enough to verify allowability 
on the simplex level because we want to identify all $\bar{p}$-allowable 
simplicial chains.
We describe an algorithm that can be used to generate the intersection chain 
complex from a given simplicial chain complex. This is taken from 
\cite{bendichharer} 
in which an algorithm was introduced to compute persistent intersection 
homology with $\intg_2$ coefficients. We will here only describe 
those parts of the algorithm that solve the problem of identifying the intersection 
chains. 

Let $K$ be a simplicial complex and $\bar{p}$ be an arbitrary perversity. 
Note that we are over $\intg_2$ and thus orientations 
are not an issue. Let $IS_k^{\bar{p}}(K)$ denote the collection of all 
$\bar{p}$-allowable $k$-simplices of $K$. Next, for every $k$, 
we choose an ordering of 
$IS_{k}^{\bar{p}}(K)$ such that 
$IS_{k}^{\bar{p}}(K) = \{{\sigma^{(k)}_0,\dots,
\sigma^{(k)}_{r_k}, \sigma^{(k)}_{r_k+1},\dots, \sigma^{(k)}_{m_{k}}}\},$ 
where $\sigma^{(k)}_0,\dots,\sigma^{(k)}_{r_k}$ are all $\bar{p}$-allowable 
$k$-simplices of $K$. 
Set up an incidence matrix $M$ over $\intg_2$ with entries
\[
\begin{cases}
M_{ij} = 1 \mbox{, if } \sigma^{(k-1)}_{i} \text{ is a face of } 
        \sigma^{(k)}_{j}, \\
M_{ij} = 0 \mbox{, else,}
\end{cases}
\]
$i=0,\ldots, m_{k-1},$
$j=0,\ldots, r_k.$
The columns that have nonzero entries 
below the $r$-th row represent $\bar{p}$-allowable $k$-simplices whose boundaries contain not $\bar{p}$-allowable $(k-1)$-simplices and therefore do not 
represent a $\bar{p}$-allowable elementary chain.
By elementary column transformations from left to right, we want to reduce 
as much of the rows below $r$ as possible to form allowable chains. This is 
done by adding up columns with the same value in rows with higher index than $r$.
The pseudocode of Algorithm \ref{alg.MatRed} describes how this reduction 
can be realized on a computer.
\begin{algorithm}
  \caption{MatrixReduction$(M,r)$}
  \begin{flushleft}
   \textbf{Input:} \text{$M$ binary incidence matrix, $r$ row index} \\
   \textbf{Output:} \text{$M$ matrix in reduced form} 
  \end{flushleft}

  \begin{algorithmic}[1]
  \State $ncol \gets \operatorname{length}(M[0,:])$ \Comment{Number of columns}
  	\For{$j :=  ncol-1$ \textbf{downto} $1$}
  		\While{$\exists i < j$ with $\mbox{low}(M,i) = \mbox{low}(M,j)$ 
  		            and low$(M,j)>r$}  
  				\State $M[:,j] \leftarrow M[:,j]+M[:,i]$
  				
		 \EndWhile
   	\EndFor 
  \State \Return{$M$}
  \end{algorithmic}
  \label{alg.MatRed}
\end{algorithm}
The function low$(M,j)$ returns the index of the lowest 
nonzero entry in the $j$-th column of a given matrix $M$, in case it exists. 
The function returns $-1$, if the column only contains zeros.
Furthermore, 
it is possible to record the column transformations executed during a 
reduction process to determine a basis for the intersection chain complex 
in every degree. Now, deleting all columns $i$ in $M$ with low$(M,i) > r$ 
and all rows below $r$, one obtains matrices whose ranks can be
used to calculate the intersection Betti numbers. 
Our Python code, both for the computation of stratified spines and
intersection Betti numbers, is available at
\url{https://github.com/BanaglMaeder/layered-spines.git}.
It was not our first priority to optimize the code for efficiency,
and certainly several refinements in this direction are possible.

\section{Examples and Evaluation}
\label{sec.examples}

We illustrate stratified formal deformations and associated
intersection homology groups by algorithmically
computing several stratified spines of Vietoris-Rips type complexes
associated to data point samples near given spaces.
Our first example concerns the cone on a circle. Such a cone is
topologically a $2$-disc and thus nonsingular, but when filtered
by the cone-vertex and using appropriate perversities, may nevertheless
have nonvanishing intersection homology in positive degrees.
\begin{example}
As an implicit surface, the cone on $S^1$ is given by
\[ \cone (S^1)=
\{ (x,y,z) \in \real^3 ~|~ x^2 + y^2
   = c^2 (z-1)^2 \mbox{ and } 0\leq  z\leq 1 \}
\]
with $c$ being the base/height ratio of the cone; the cone vertex
is $s=(0,0,1)$. When filtered by the cone vertex of codimension $2$,
the cone has intersection homology
\begin{align*}
IH^{\bar{0}}_{i} (\cone S^1;\intg_2) &= 
\begin{cases}
\intg_2, &\mbox{ if } i= 0\\
0, &\mbox{ else.}
\end{cases}
\\
IH^{\overline{-1}}_{i} (\cone S^1;\intg_2) &= 
\begin{cases}
\intg_2, &\mbox{ if } i=0,1 \\
0, &\mbox{ else.}
\end{cases}
\end{align*}
(Recall that we write $\intg_2 = \intg/2\intg$.)
Choosing $c=3$,
we then produced a set of 34 random points in $\real^3$ 
by rejection sampling. In more detail, independent sampling of two uniformly 
distributed points on the interval $[-1,1]$ for the $x$- and $y$-coordinate 
and on the interval $[0,1]$ for the z-coordinate yields a triple 
$(x,y,z) \in \real^3$. If such a point satisfied the equation that defines 
the cone up to an error of $0.001$, we accepted 
that point as part of our sample. The cone vertex $s$ is part of the 
sample as well.
Applying Vietoris-Rips type methods, we 
generated a $3$-dimensional simplicial complex $K$ with polyhedron $X=|K|$
depicted in Figure \ref{fig.sampCone}. 
Taking $S^0 = \{ s \}$ endows $K$ with the structure of a 
divided complex.
Using the natural geometric codimension $\codim \{ s \}=3$ in the
$3$-dimensional polyhedron $X$, the intersection homology 
is given by
\begin{align*}
IH^{\bar{0}}_{i} (X;\intg_2) &= 
\begin{cases}
\intg_2, &\mbox{ if } i=0,1 \\
0, &\mbox{ else.}
\end{cases}
\\
IH^{\overline{-1}}_{i} (X;\intg_2) &= 
\begin{cases}
\intg_2, &\mbox{ if } i=0,1 \\
0, &\mbox{ else,}
\end{cases}
\end{align*}
which does not agree with the above intersection homology of the
underlying cone.

A stratified spine $(X',\Sigma' = S^0$ of $(X,\Sigma = S^0)$ is shown 
in Figure \ref{fig.sampConeLSpine}.
Using the algorithms described in Section \ref{sec.implementation},
this stratified spine was obtained on a computer by carrying out
$93$ elementary $C$-collapses and $4$ elementary intermediate collapses.
Note that the stratified spine is in fact stratum-preserving
homeomorphic to the cone on $S^1$, which implies in particular 
that the intersection homology of the stratified spine agrees with
the above intersection homology of $\cone (S^1)$
(which can also be algorithmically verified.)
Performing further unrestricted collapses on the stratified spine would
show that an ordinary spine of $X$ is a point. This shows that $X$
is (simple) homotopy equivalent to the cone on $S^1$. So the ordinary
spine does not preserve enough structure to carry the correct
intersection homology.
The stratified spine contains substantially fewer simplices than
$K$ and so the linear algebra packages performing rank computations
to determine the intersection homology can operate on substantially
smaller matrices. We find that computing spines first is often a way to
avoid sparse matrix techniques altogether.
In light of our invariance results, Corollaries
\ref{cor.stratifiedspines} and \ref{cor.ihstable},
the discrepancy between the intersection homology of the
Vietoris-Rips polyhedron $X$ and the intersection homology of
the stratified spine is due to the discrepancy between the codimensions
of the singular point. Thus computing stratified spines is in particular
a way of obtaining better estimates of correct codimensions.
\begin{figure}
\centering
    \begin{minipage}{0.45\textwidth}
        \centering
        \includegraphics[width=1\textwidth]{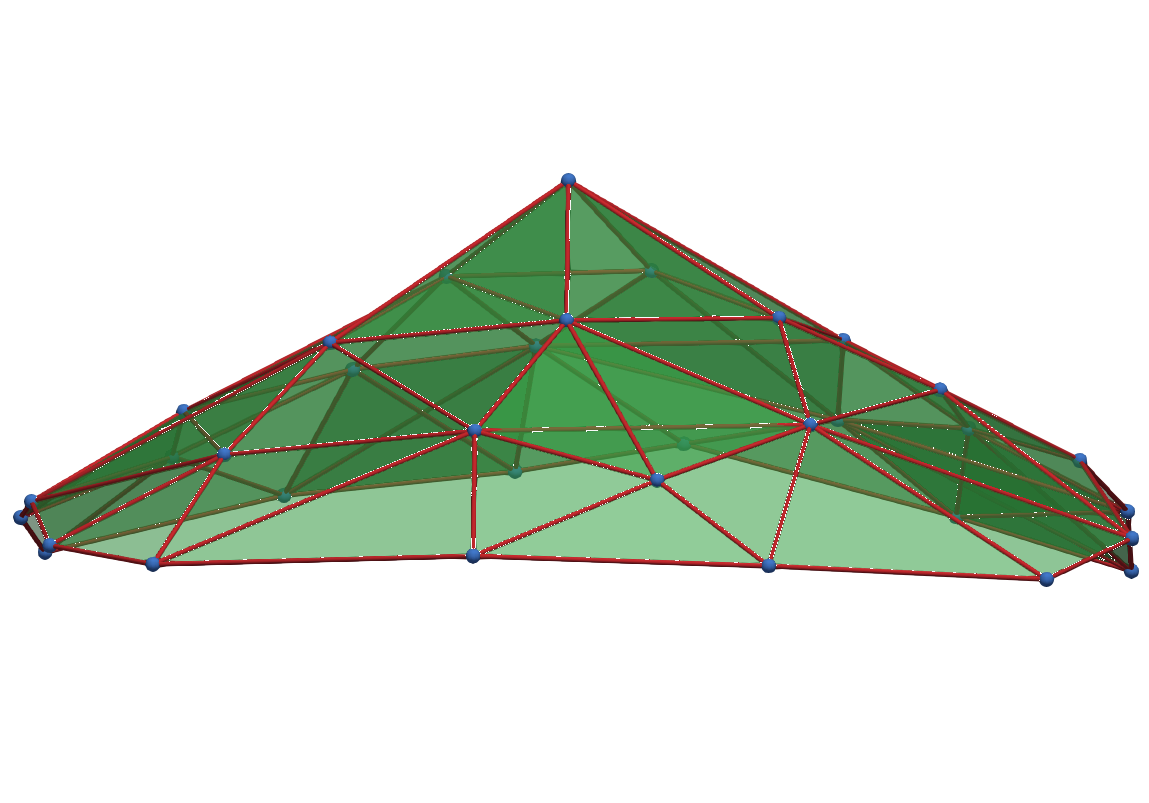}
		\put(-100,95){Cone point}
		\caption{\mbox{Sampled cone $X$.}} 
		\label{fig.sampCone}
    \end{minipage}\hfill
    \begin{minipage}{0.45\textwidth}
        \centering
        \includegraphics[width=1\textwidth]{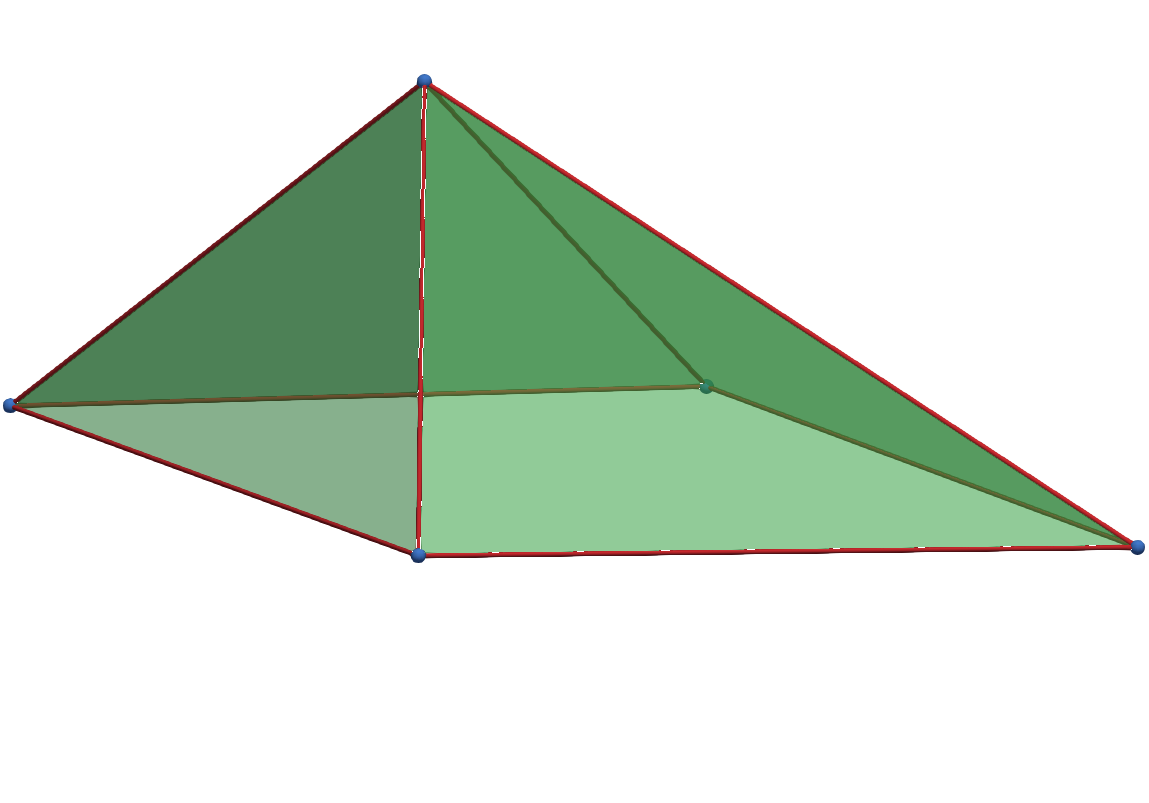}
		\put(-123,109){Cone Point}
		\caption{\mbox{Stratified spine $X'$.}}
		\label{fig.sampConeLSpine}
    \end{minipage}
\end{figure}
\end{example}

The next example deals with a much larger dataset on a pinched torus.
\begin{example}
Topologically, the pinched torus is the quotient space of the
torus $T^2 = S^1 \times S^1$ obtained by collapsing 
a circle $\pt \times S^1$. The image of this circle is the singular
point $s$ of the pinched torus.
Points were sampled from an embedded pinched torus. 
We sampled independently a pair of identically uniformly 
distributed points on the intervall $[0,2\pi]$. Then a parametrization of 
the surface was used to transform the sample to lie on the pinched torus. 
Applying Vietoris-Rips methods resulted in a simplicial complex $K$ with
$543$ $0$-simplices, $2109$ $1$-simplices, $2057$ $2$-simplices 
and $490$ $3$-simplices. 
Figure \ref{fig.PT2VR} shows the polyhedron $X=|K|$ of this 
complex. Figure \ref{fig.PT2VRVol} 
shows the volume of all $490$ 3-simplices. 
Putting $S^0 = \{ s \}$ endows $K$ with the structure of 
a divided simplicial complex. The polyhedron $X$ of the associated 
layered complex $(K,C,S=S^0)$ is filtered by $\Sigma = \{ s \}$.
A stratified spine $X'$ of $X$ was
determined algorithmically and is displayed in 
Figure \ref{fig.LaySpine}. This computation comprised
$978$ elementary $C$-collapses and $11$ elementary intermediate collapses.
Figure \ref{fig.InterLaySpine} shows an intermediate stage in the process of 
determining the stratified spine. At that point, all $C$-collapses 
have been executed and the next step is collapsing all intermediate 
simplices as far as possible. After the whole process, \emph{all} 
$3$-simplices initially present have been removed and the overall size of the 
simplicial complex has been reduced to
$537$ $0$-simplices, $1610$ $1$-simplices and $1074$ $2$-simplices. 
The stratified spine $X'$ is filtered by 
$|S^0| = X^0 \subset X^2 = X'$ and therefore simplicial 
intersection homology groups are defined. Using our implementation,
we obtain for perversity $\bar{0}$ simplicial intersection homology 
of the stratified spine:
\[
IH^{\bar{0}}_{i}(X';\intg_2) =
\begin{cases}
 \intg_2, &\mbox{ if } i= 0,2 \\
 0, &\mbox{ else.}
\end{cases} 
\]
These results agree with the singular perversity $\bar{0}$
intersection homology of the actual pinched torus.
They also agree with the intersection homology of the full
Vietoris-Rips polyhedron $X$, \emph{both} when $\{ s \}$ is
assigned formal codimension $3$ (the geometric codimension in $K$)
and when it is assigned formal codimension $2$ (the geometric codimension
in the stratified spine).

	\begin{figure}
    \centering
	\includegraphics[width=0.8\textwidth]{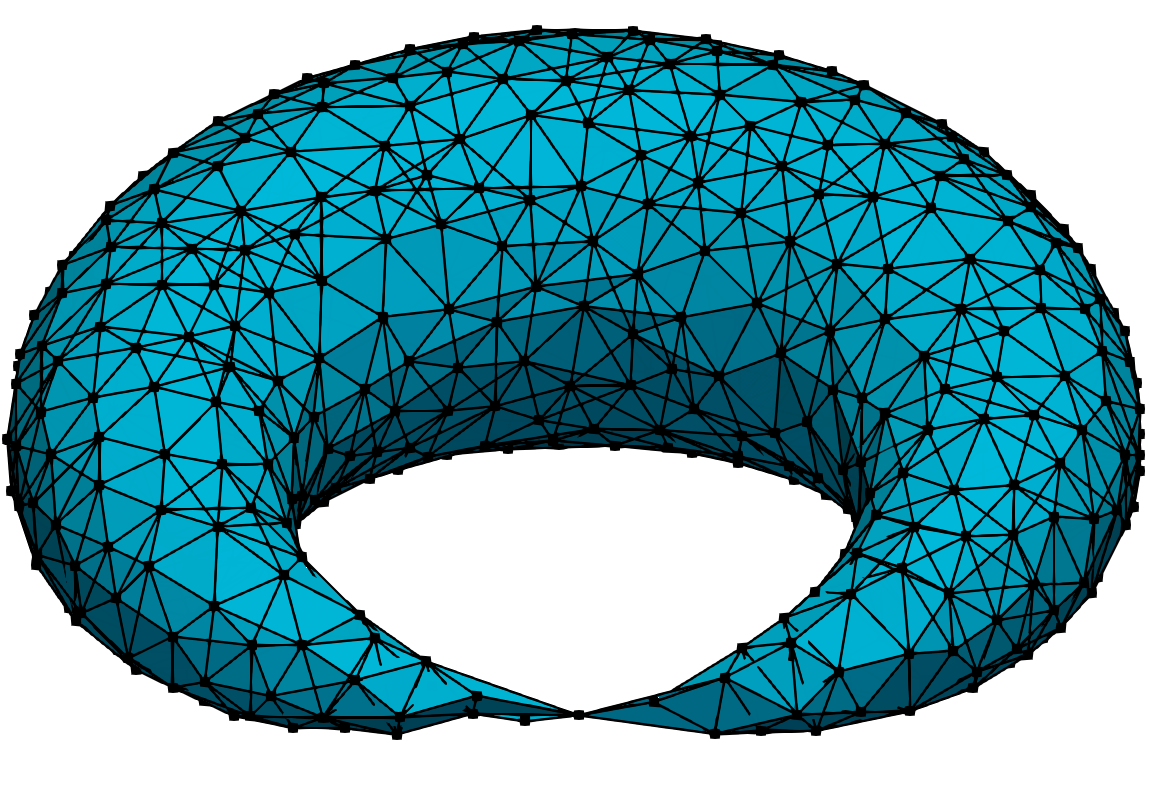}
	\put(-150,27){$s$}
	\caption{Simplicial complex from Vietoris-Rips methods}
	\label{fig.PT2VR}
	\end{figure}
	\begin{figure}[H]
	\centering
	\includegraphics[width=0.8\textwidth]{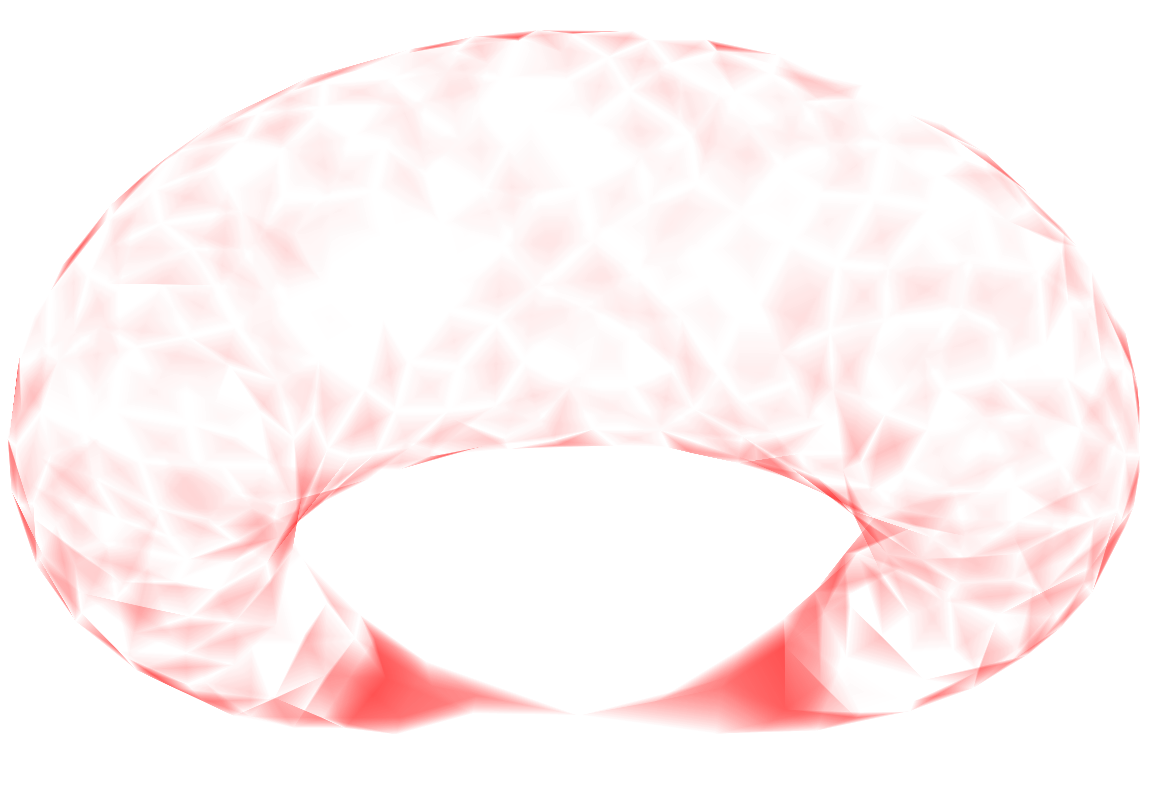}
	\caption{Volume extracted from 3 simplices}
	\label{fig.PT2VRVol}
	\end{figure}
	\begin{figure}[H]
	\centering
	\includegraphics[width=0.8\textwidth]{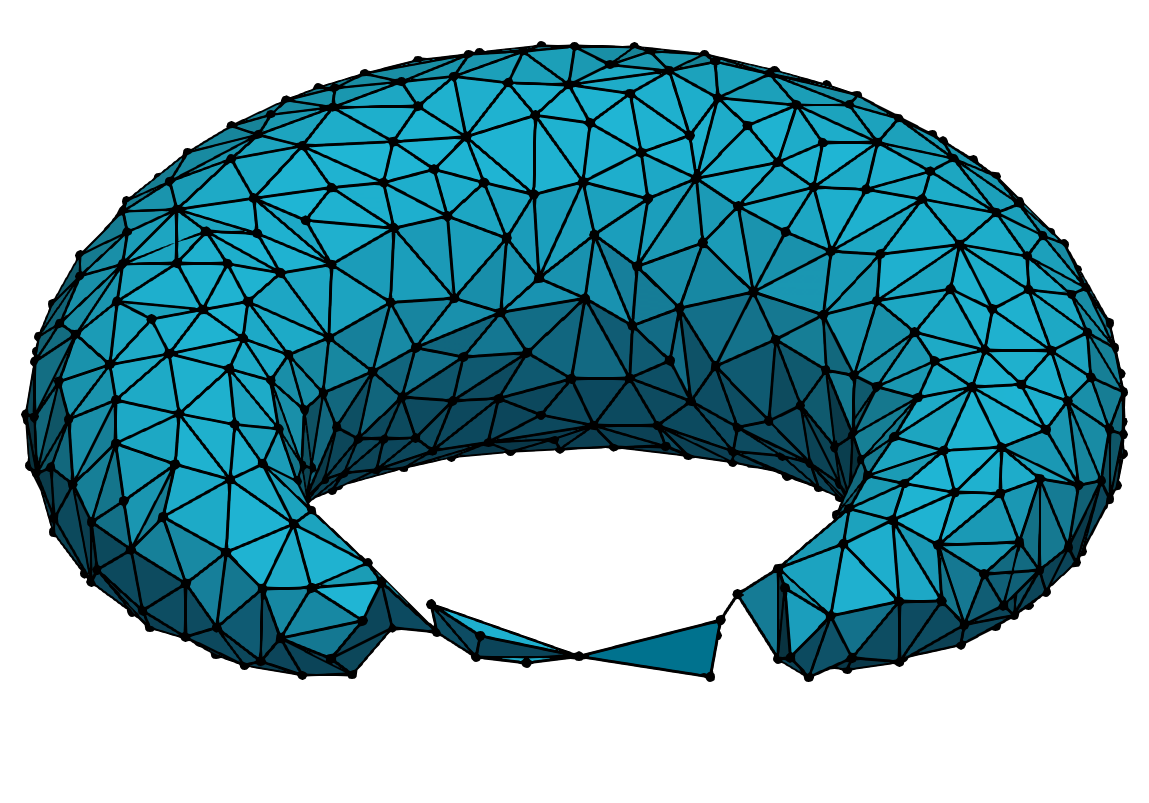}
	\put(-150,42){$s$}
	\caption{Intermediate stage in the collapse process}
	\label{fig.InterLaySpine}
	\end{figure}
	\begin{figure}[H]
	\centering
	\includegraphics[width=0.8\textwidth]{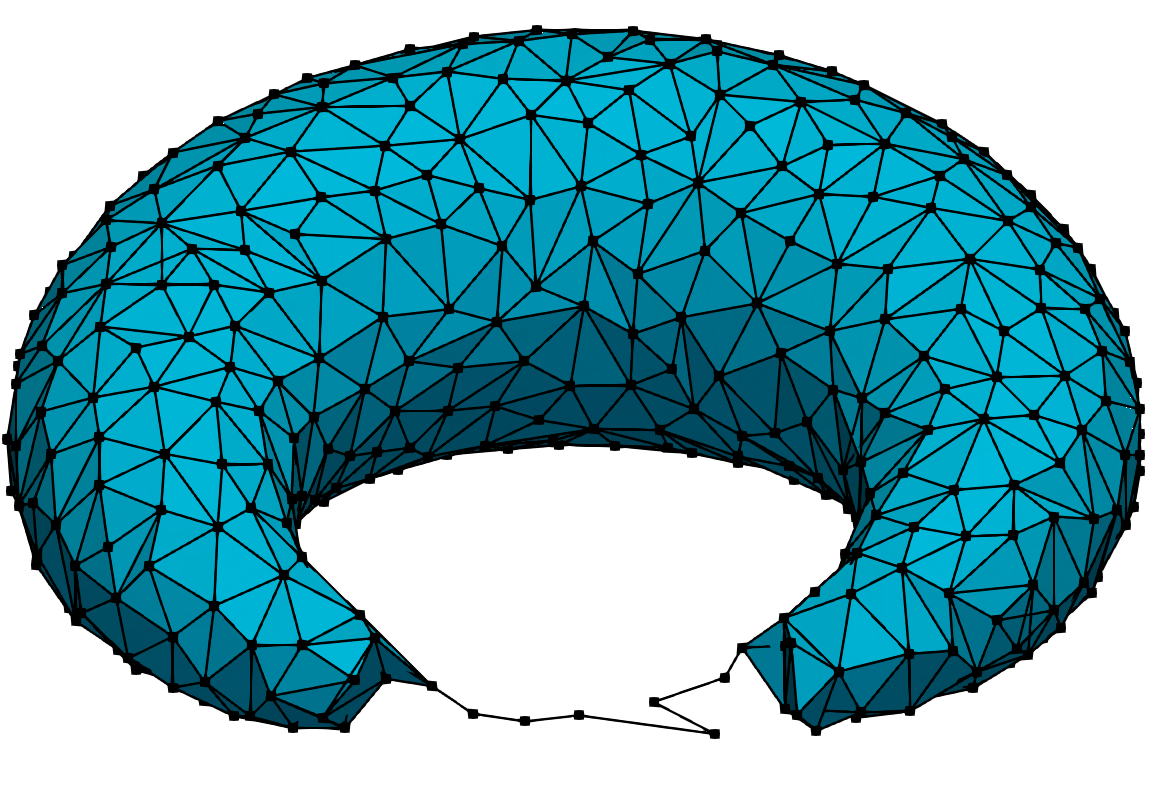}
	\put(-150,27){$s$}
	\caption{Stratified Spine.}
	\label{fig.LaySpine}
	\end{figure}
\end{example}

\begin{example}
We may return to the figure eight space $S^1 \vee S^1$
of Example \ref{expl.s1vs112pts} and reconsider it from the
stratified spinal point of view.
The (ordinary) spine $X'$ shown in Figure
\ref{fig.S1vS1VR12ptsSpine} is in fact also
a stratified spine of $X$.
The green $2$-simplex in the lower left of Figure
\ref{fig.S1vS1VR12pts} is removed by an elementary $C$-collapse,
while the other two green $2$-simplices, incident to the singular
point, are removed by elementary intermediate collapses.
So here the stratified spine happens to be stratified homeomorphic
to the ordinary spine and both compute the intersection
homology of $S^1 \vee S^1$.
\end{example}

Finally, it is clear that the Vietoris-Rips
method (and similar methods) may create local thickenings near
the singularities
that cannot be removed by stratified collapses
to an extent that would make the stratified homotopy type of the
original space, or at least its intersection homology, visible.
An example polyhderon $X$ 
obtained by sampling two tangent circles, topologically comprising
$S^1 \vee S^1$, is shown in 
Figure \ref{fig.S1vS1VR}. 
Finding general characterizations of such
configurations, as well as, perhaps even more importantly, improving
the Vietoris-Rips method in stratified situations,
is an interesting question that requires methods outside the scope
of the present paper.
\begin{figure}
	\includegraphics[width=0.9\textwidth]{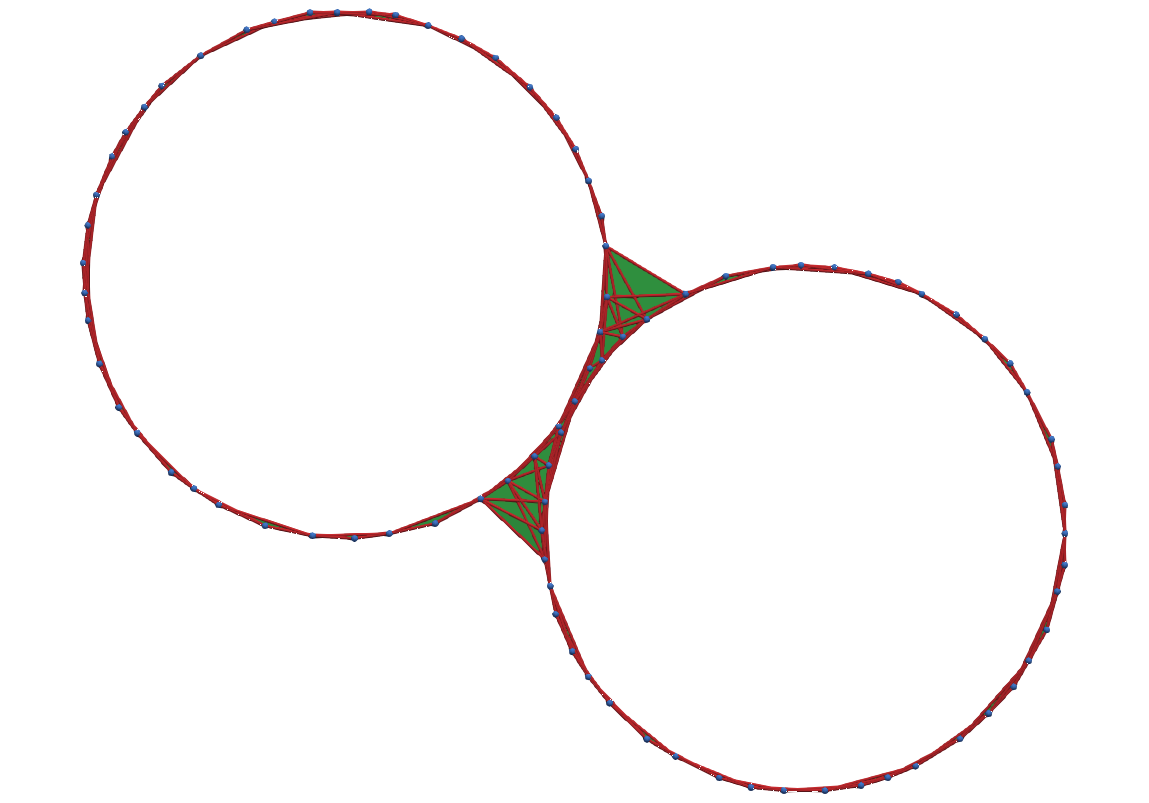}
	\put(-173,115){$s$}
	\caption{Tangent circles.}
	\label{fig.S1vS1VR}
\end{figure}
It is perhaps interesting to observe that
$X$ contains an \emph{ordinary} spine  $X'$,
shown in Figure \ref{fig.S1vS1Spine}, which is stratified homeomorphic to
$S^1 \vee S^1$.
\begin{figure}
	\includegraphics[width=0.9\textwidth]{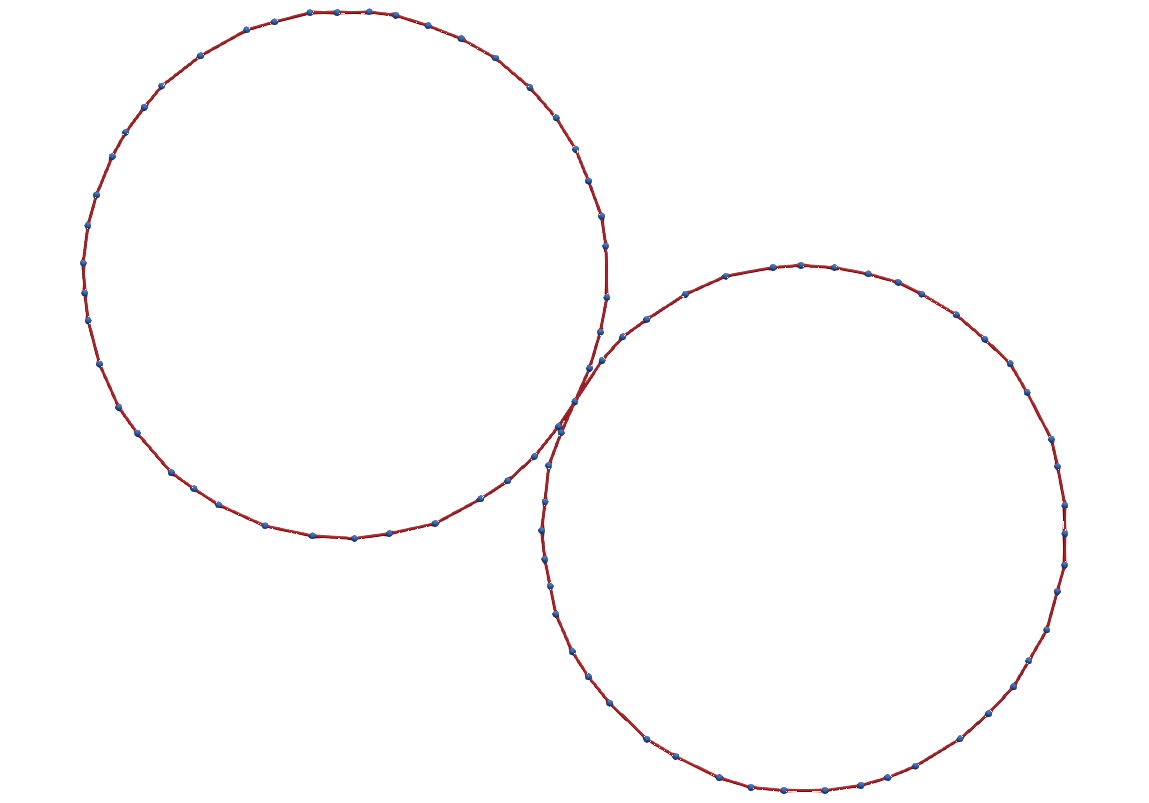}
	\put(-173,115){$s$}
	\caption{Ordinary Spine $X'$ of $X$.}
	\label{fig.S1vS1Spine}
\end{figure}
One could interpret this as an indication that there exist 
further types of
intermediate collapses. These may be formulated and explored in
future work. Furthermore, if one is merely interested in
intersection homology $IH^{\bar{p}}_*$ for a particular perversity
$\bar{p}$, then one might expect to obtain a relaxed perversity-dependent version of
the intermediate collapse condition (iv).

\end{document}